\documentclass[11pt,reqno,a4paper]{amsart}
	\usepackage{amsmath,amssymb,amsthm,framed,cancel,tikz,hyperref,enumitem}

\newtheorem{theorem}{Theorem}[section]
\newtheorem{lemma}[theorem]{Lemma}

\newtheorem{proposition}[theorem]{Proposition}

\theoremstyle{definition}
\newtheorem{definition}[theorem]{Definition}

\theoremstyle{remark} 
\theoremstyle{remark} 
\theoremstyle{remark}

\theoremstyle{remark}
\newtheorem{remark}[theorem]{Remark}

\numberwithin{equation}{section}

\begin{document}
\title[Few Bilinear Operators on Spaces of Continuous Functions]{Few Bilinear Operators on Spaces of Continuous Functions}

\author{Leandro Candido}
\address{Universidade Federal de S\~ao Paulo - UNIFESP. Instituto de Ci\^encia e Tecnologia. Departamento de Matem\'atica. S\~ao Jos\'e dos Campos - SP, Brasil}
\email{\texttt{leandro.candido@unifesp.br}}
\thanks{ The author was supported by Funda\c c\~ao de Amparo \`a Pesquisa do Estado de S\~ao Paulo - FAPESP No. 2023/12916-1 }




\begin{abstract}
Motivated by recent work exhibiting a locally compact scattered space $L$ constructed under Ostaszewski's $\clubsuit$-principle, which yielded a complete classification of linear operators on $C_0(L\times L)$, we extend the analysis to the bilinear setting. We show that, for this space $L$, every bilinear operator $G:C_0(L)\times C_0(L)\to C_0(L)$ admits a unique decomposition into the sum of trivially predictable components. This establishes a bilinear analogue of the \emph{few operators} phenomenon.
\end{abstract}

\maketitle

\section{Introduction}

A problem that has inspired numerous research directions over the past decades concerns the possible linear operators on a Banach space; see, for instance, \cite{Argyros2011}, \cite{GowersMaurey}, \cite{Lindenstrauss}, \cite{Mor}, \cite{Shelah}, among many others. This problem is particularly intriguing in the context of Banach spaces of continuous functions on a compact Hausdorff space $C(K)$; see \cite{Koz2}, \cite{Koz3}, \cite{Pleb}, and even in the case where $K$ is scattered; see \cite{Koz1}, \cite{Koz4}, and \cite{Koz5}.  

The inspiration for this work comes from the recent construction presented in \cite{CandidoC(KxK)}, under the assumption of Ostaszewski’s $\clubsuit$-principle \cite{Osta}, which produced a locally compact Hausdorff space $L$ allowing the classification of all operators on $C_0(L\times L)$; see \cite[Theorem~1.1]{CandidoC(KxK)}. The underlying intuition is that the exotic properties of this space $L$ also lead to the failure of continuity for bilinear operators of the form  
\[
G:C_0(L)\times C_0(L)\to C_0(L).
\]  
Indeed, for each fixed function $f\in C_0(L)$, the formula $G_f(g)(t)=G(f,g)(t)$ defines a linear operator which, due to the properties of $L$, takes the simplest possible form. This naturally motivates the investigation of how the exotic topological characteristics of $L$ influence the continuity of bilinear mappings of the above type.  

The first step in this investigation is to identify, among the possible bilinear operators on $C_0(L)$, beyond the obvious ones with separable range, the most elementary types. In our setting, two such types appear. The first and most basic is what we refer to as \emph{scaled multiplication}, given by  
\[
M(f,g)(t) = p\, f(t)g(t),
\]  
where $p$ is a fixed constant.  

The second arises by fixing $(a_w)_{w\in L}, (b_w)_{w\in L}\in \ell_1(L)$ and defining what we call \emph{cross-interaction}, given by  
\[
H(f,g)(t)=\sum_{w\in L}\left(a_w f(t)g(w)+b_w f(w)g(t)\right).
\]  

In this paper we show that every bilinear operator of the form $G:C_0(L)\times C_0(L)\to C_0(L)$ can be decomposed into the sum of three components: a scaled multiplication component, a cross-interaction component, and a separable-range component. Since operators of these types exist regardless of the topology of $L$, this establishes a notion of \emph{few operators} in the bilinear setting. More precisely, the main result of this work is the following:

\begin{theorem}\label{Thm:Main}
Assuming Ostaszewski's principle $\clubsuit$, there exists a non-metrizable locally compact scattered Hausdorff space $L$ such that, for every bilinear operator $G:C_0(L)\times C_0(L)\to C_0(L)$, there exist a unique $r\in \mathbb{R}$, unique 
$(a_w)_{w\in L}, (b_w)_{w\in L} \in \ell_1(L)$, and a unique bilinear operator with separable range $S:C_0(L)\times C_0(L)\to C_0(L)$ such that 
\[
G(f,g)(t) = r\, f(t)g(t) 
           +\sum_{w\in L}(a_w  f(t)g(w)+ b_w f(w)g(t)) 
           + S(f,g)(t),
\]
for all $f,g \in C_0(L)$ and $t\in L$.
\end{theorem}

This paper is organized as follows. In Section~\ref{Sec:Basic}, we establish the basic terminology and recall the notion of the \emph{collapsing space}, along with its exotic properties that will be exploited throughout the paper. In Section~\ref{Se:Frechet}, we present fundamental results on Fréchet measures and discuss their intrinsic connection with bilinear forms on $C_0(L)$. Section~\ref{Sec:Decomposition} contains the main technical part of the paper, which reveals the pathologies of many bilinear operators on $C_0(L)$. Finally, in Section~\ref{Sec:Final}, we provide a detailed proof of Theorem~\ref{Thm:Main}.

\section{Terminology and Preliminaries}
\label{Sec:Basic}

We begin by introducing some basic terminology that will be used throughout this work.  
All ordinals are equipped with the usual order topology.  
As is customary, we denote by $\omega$ the first infinite ordinal and by $\omega_1$ the first uncountable ordinal.  

We use the letters $L$ and $K$ to denote (locally) compact Hausdorff spaces.  
In particular, whenever $L$ is locally compact, we let $K=L\cup \{\infty\}$ denote its one-point (Aleksandrov) compactification.  

For such spaces $L$ and $K$, we denote by $C_0(L)$ the Banach space of continuous real-valued functions on $L$ vanishing at infinity, and by $C(K)$ the Banach space of continuous real-valued functions on $K$, both equipped with the supremum norm. Throughout this paper, we always regard $C_0(L)$ as the closed subspace of $C(K)$ consisting of functions $f:K\to\mathbb{R}$ with $f(\infty)=0$.  
For functions $f, g \in C(K)$, we denote by $f \otimes g : K \times K \to \mathbb{R}$ the function defined by $f \otimes g (x,y) = f(x) g(y)$.

For a compact Hausdorff space $K$, we denote by $\mathfrak{B}(K)$ its Borel $\sigma$-algebra, and by $\wp(K)$ its power set. By the Riesz representation theorem \cite[Theorem 18.4.1]{Se}, the topological dual of $C(K)$ is isometrically identified with $\mathcal{M}(K)$, the space of Radon measures on $K$ endowed with the total variation norm. If $\mu\in \mathcal{M}(K)$, we denote by $|\mu|$ its total variation. The dual of $C_0(L)$ is identified with the subspace $\mathcal{M}_0(K) = \{\mu \in \mathcal{M}(K) : \mu(\{\infty\})=0\}$.

At the core of this study lies a locally compact Hausdorff space $L$ constructed in \cite[Theorem~3.6]{CandidoC(KxK)} under Ostaszewski's principle $\clubsuit$. We refer to this space as the \emph{collapsing space} $L$ and recall the properties that will be relevant throughout this paper.  

We recall the class $\mathcal{S}$ of topological spaces, see \cite{CandidoC(KxK)} and \cite{candido2}. A locally compact Hausdorff space $L$ belongs to the class $\mathcal{S}$ if there exists a continuous finite-to-one surjection $\varphi:L\to\omega_1$, that is, $\varphi^{-1}(\{\alpha\})$ is finite for every $\alpha\in\omega_1$. Consequently, $L$ is a first-countable scattered space \cite[Proposition~3.2]{candido2}. We always consider such a space $L$ together with the collection $\{L_\alpha : \alpha\in\omega_1\}$, where $L_\alpha = \varphi^{-1}[\alpha+1] = \{x\in L : \varphi(x)\in \alpha+1\}$.

This collection forms a covering of $L$ by countable clopen sets and satisfies $L_\alpha\subset L_\beta$ whenever $\alpha\leq \beta$. In particular, every function $f\in C_0(L)$ has countable support. For each $\alpha<\omega_1$, we identify $C_0(L_\alpha)$ with the subspace of $C_0(L)$ consisting of functions that vanish outside $L_\alpha$. With this, we can establish the following simple but important fact.

\begin{proposition}\label{Prop:SeparableImage}
Let $L$ be a member of the class $\mathcal{S}$. If $W$ is a separable subset of $C_0(L)$, then there exists $\eta<\omega_1$ such that $X\subset C_0(L_\eta)$.
\end{proposition}
\begin{proof}
Let $W$ be a separable subset of $C_0(L)$, and set $X=\overline{\mathrm{span}}(W)$. Then $X$ is a separable subspace of $C_0(L)$. Fix a dense sequence $\{f_n\}_{n<\omega}$ in $X$. Since each $f_n$ has countable support, we can choose $\eta_n$ such that $f_n$ vanishes on $L\setminus L_{\eta_n}$. Let $\eta=\sup\{\eta_n:n<\omega\}$. Hence every $g\in X$ vanishes outside $L_\eta$.
\end{proof}

Next, we recall the following terminology from \cite{CandidoC(KxK)}. 
If $\{I_1, \ldots, I_k\}$ is a partition of $\{1, 2, \ldots, m\}$, 
a collection $\mathcal{C}=\{(x_1^\alpha, \ldots, x_m^\alpha) : \alpha < \omega_1\} \subset L^m$ 
is said to be \emph{$(I_1, \ldots, I_k)$-separated} if the family 
$\{(A_1^\alpha, \ldots, A_k^\alpha) : \alpha < \omega_1\}$, where 
$A_j^\alpha = \{x_i^\alpha : i \in I_j\}$ for each $j=1, \ldots, k$, satisfies  
\[
A_i^\alpha \cap A_j^\beta = \emptyset \quad \text{whenever } i \neq j \text{ or } \alpha \neq \beta.
\]

\begin{definition}\label{Def:Collapsing}
A locally compact space $L$ is called \emph{collapsing} if it belongs to the class $\mathcal{S}$ and, 
for every integer $1 \leq m < \omega$, every partition $\{I_1, \ldots, I_k\}$ of $\{1, \ldots, m\}$, 
and every $(I_1, \ldots, I_k)$-separated collection 
$\{(x_1^\lambda, \ldots, x_m^\lambda) : \lambda \in \varGamma\} \subset L^m$, 
there exists an accumulation point $(q_1, \ldots, q_m) \in K^m$ such that there exist distinct points 
$a_1, \ldots, a_k \in K \cup \{\infty\}$, with $a_k = \infty$, such that  
\[
\{q_i : i \in I_j\} = \{a_j\} \quad \text{for every } 1 \leq j \leq k.
\]
\end{definition}

We recall the following result from \cite[Theorem~3.6]{CandidoC(KxK)}, which provides the collapsing space central to our study.
\begin{theorem}\label{Thm:Collapsing}
Under Ostaszewski’s $\clubsuit$ principle, there exists a collapsing space.
\end{theorem}

\begin{remark}
It should be emphasized that collapsing spaces cannot be constructed within \textbf{ZFC} alone, as their existence requires additional set-theoretic assumptions; see \cite[Remark~3.8]{CandidoC(KxK)}. However, it is not excluded that a version of Theorem~\ref{Thm:Main} might hold under other combinatorial principles, as in \cite{Koz5}, or even within \textbf{ZFC} itself, as in \cite{Koz1}.
\end{remark}

The following result is a streamlined for our purposes version of the combination of \cite[Theorem 4.3 and Corollary 4.6]{CandidoC(KxK)}. We recall that $\nu:K\to \mathcal{M}_0(K)$ is said to be  a weak$^*$-continuous function vanishing at infinity if $\nu$ is continuous, when $\mathcal{M}_0(K)$ is endowed with the weak$^*$-topology induced by $C_0(L)$ and such that $\nu(\infty)=0$, see \cite[Theorem 4.1]{CandidoC(KxK)}.

\begin{proposition}\label{Thm:AuxiliarCallapse2}
Let $\nu:K\to \mathcal{M}_0(K)$ be a weak$^*$-continuous function vanishing at infinity.  
Then there exist $r\in\mathbb{R}$ and $\rho<\omega_1$ such that, for every $x,y\in L\setminus L_\rho$, one has
\[
\nu_{x}(\{y\})=
\begin{cases}
r, & \text{if } x = y, \\
0, & \text{if } x \neq y.
\end{cases}
\]
\end{proposition}

To conclude this section, we introduce some additional terminology that will be used throughout the paper.

\begin{definition}\label{Def:Refinement}
Let $\{x_\alpha\}_{\alpha<\omega_1}$ be a transfinite sequence of objects.  
A \emph{refinement} (or \emph{refined sequence}) of this family is any subsequence 
$\{x_{\alpha_\gamma}\}_{\gamma<\omega_1}$, possibly selected to satisfy additional properties. More generally, given several transfinite sequences $\{x^1_\alpha\}_{\alpha<\omega_1},\ldots,\{x^m_\alpha\}_{\alpha<\omega_1}$, a \emph{simultaneous refinement} is obtained by choosing a strictly increasing sequence of ordinals $\{\alpha_\gamma\}_{\gamma<\omega_1}$ and considering the corresponding subsequences
$\{x^1_{\alpha_\gamma}\}_{\gamma<\omega_1}, \ldots, \{x^m_{\alpha_\gamma}\}_{\gamma<\omega_1}$. Whenever no ambiguity arises, we continue to denote these refined sequences by the original symbols $\{x^1_\alpha\}_{\alpha<\omega_1}, \ldots, \{x^m_\alpha\}_{\alpha<\omega_1}$.
\end{definition}

All remaining terminology used in this paper can be found in \cite{CandidoC(KxK)}.

\section{Fréchet Measures and the Structure of Bilinear Forms on $C(K)$}
\label{Se:Frechet}
In this section, we collect several results concerning the interplay between bilinear forms on $C(K)$ and the so-called Fr\'echet measures on $K \times K$, focusing on the particular case where $K$ is a scattered compact Hausdorff space. Our main reference for this material is the comprehensive monograph by Ron Blei \cite[Chapter IV and Chapter VI]{Blei}, from which we adopt much of the basic terminology and adapt several results to our setting.

\begin{definition}
An \emph{$F_2$-measure} (or \emph{2–Fr\'echet measure}) on $K \times K$ is a function 
$\mu : \mathfrak{B}(K) \times \mathfrak{B}(K) \longrightarrow \mathbb{R}$ such that 
$\mu$ is a measure in each coordinate separately. That is, for every fixed 
$A \in \mathfrak{B}(K)$, both mappings $B \mapsto \mu(A,B)$ and $B \mapsto \mu(B,A)$ 
are (signed) measures on $\mathfrak{B}(K)$. 

We denote by $\mathcal{F}(K,K)$ the collection of all such $F_2$-measures on $K \times K$.  If $K = L \cup \{\infty\}$ is the one-point compactification of $L$, we denote by $\mathcal{F}_0(K,K)$ the subspace of $\mathcal{F}(K,K)$ consisting of all $F_2$-measures 
$\mu$ such that the mappings $B \mapsto \mu(\{\infty\},B)$ and $B \mapsto \mu(B,\{\infty\})$ are identically zero.
\end{definition}

For the benefit of readers not familiar with $F_2$-measures, we briefly outline the concept of integration with respect to a measure $\mu \in \mathcal{F}(K,K)$, following Blei's notation \cite[Chapter IV, Section 5]{Blei}. First, for every function $g \in C_0(L)$, 
we denote by $\mu_g$ the measure on $K$ defined by the formula
\[
\mu_g(A) = \int g \, d\mu(\cdot, A).
\]

For functions $f,g \in C(K)$, the iterated integral of $f \otimes g$ with respect to $\mu$ is given by
\[
\int f \otimes g \, d\mu 
   = \int_{K} f(x) \, d\mu_g=  \int_{K} f(x) \left( \int_{K} g(y)\,\mu(\cdot, dy) \right)(dx).
\]

It is important to note (see \cite[Chapter VI, Theorem~10]{Blei}) that the following Fubini-type property holds:
\begin{align*}
\int f \, d\mu g
   &= \int_{K} f(x) \left( \int_{K} g(y)\,\mu(\cdot, dy) \right)(dx) \\
   &= \int_{K} g(y) \left( \int_{K} f(x)\,\mu(dx, \cdot) \right)(dy) = \int g \, d\mu f.
\end{align*}

The space $\mathcal{F}(K,K)$ can be equipped with a norm under which it becomes a Banach space. This norm is given by the Fr\'echet variation, which we recall below; see \cite[Chapter VI, Section 3]{Blei}. 
 
We say that a collection $\varLambda \subset \mathfrak{B}(K) \times \mathfrak{B}(K)$ is a \emph{grid} for $K \times K$ if $\varLambda = \varLambda_1 \times \varLambda_2$, where each $\varLambda_i$ is a finite collection of pairwise disjoint Borel subsets of $K$. If $\varLambda = \varLambda_1 \times \varLambda_2$ is a grid for $K \times K$, we denote by $\{ r_A : A \in \varLambda_i \}$ the Rademacher system on $\{-1,1\}^{\varLambda_i}$, that is, $r_A:\{-1,1\}^{\varLambda_i} \to \{-1,1\}$ is defined by $r_A(h)=h(A)$ for $h \in \{-1,1\}^{\varLambda_i}$, for $i=1,2$.

Now, if $\mu$ is a Fr\'echet measure in $\mathcal{F}(K,K)$, the \emph{Fr\'echet variation} is given by 
\[
\|\mu\|_{\mathcal{F}_2}=\sup\left\{\, \Big\|\sum_{(A,B)\in \varLambda}
   \mu(A,B)\, r_{A}\otimes r_{B}\Big\|_\infty
   : \varLambda \text{ is a grid for } K\times K
\right\}.
\]
According to \cite[Chapter IV, Theorem~5]{Blei}, we have $\|\mu\|_{\mathcal{F}_2}<\infty$. 

Since our focus is on the case where $K$ is a scattered compact space, the next definition introduces an important object that will play a fundamental role in our study.  
To simplify the notation, we shall adopt the following convention:  for $\mu \in \mathcal{F}(K,K)$ and $x,y \in K$, we write $\mu(x,y) = \mu(\{x\},\{y\})$.  Similarly, if $\mu \in \mathcal{M}(K)$, we write $\mu(x)= \mu(\{x\})$ for every $x \in K$.

\begin{definition}
Let $K$ be a scattered compact Hausdorff space and let $\mu\in \mathcal{F}(K,K)$. The \emph{total variation} of $\mu$ is the function $\mathrm{Var}(\mu):\wp(K)\times \wp(K)\to [0,\infty)$ defined by
\[\mathrm{Var}(\mu)(U,V)=\sup\Biggl\{ \sum_{x\in S}\sum_{y\in T} \mu(x,y)\,u_x v_y : T\times S \subset U\times V \ \text{finite},\ u_x,v_y\in[-1,1] \Biggr\}.\]
\end{definition}
We first note that $\mathrm{Var}(\mu)$ is well-defined, since by an argument similar to that of \cite[Chapter I, Lemma~8]{Blei} one has 
\[
\mathrm{Var}(\mu)(U,V) \le \|\mu\|_{\mathcal{F}_2} 
\quad \text{for every } (U,V)\in \wp(K)\times \wp(K).
\]
Although $\mathrm{Var}(\mu)$ does not, in general, belong to $\mathcal{F}(K,K)$, 
it enjoys several properties that will be fundamental in our analysis. 
For instance, it is subadditive: if $U, V_1,\ldots, V_n \in \wp(K)$ with $V_i\cap V_j=\emptyset$ for $i\neq j$, then  
\[
\mathrm{Var}(\mu)\!\left(U,\bigcup_{i=1}^n V_i\right)\leq \sum_{i=1}^n\mathrm{Var}(\mu)(U,V_i),
\quad
\mathrm{Var}(\mu)\!\left(\bigcup_{i=1}^n V_i,U\right)\leq \sum_{i=1}^n\mathrm{Var}(\mu)(V_i,U).
\]

It is also monotone: if $U_1\subset U_2$ and $V_1\subset V_2$, then
\[
\mathrm{Var}(\mu)(U_1,V_1) \le \mathrm{Var}(\mu)(U_2,V_2).
\]

A key fact for our research is the close relationship between $F_2$-measures on $K \times K$ and bilinear forms on $C(K)$. The following result provides a Riesz-type representation theorem for bilinear forms on $C(K)$ when $K$ is a scattered compact space.

\begin{theorem}\label{Thm:RieszBilinearScattered}
Let $K$ be a Hausdorff scattered compact space. Then, for every bilinear form 
$G$ on $C(K)\times C(K)$, there exists a unique $\mu \in \mathcal{F}(K,K)$ such that, for all $f,g \in C(K)$,
\begin{equation}\label{Rel:01}
G(f,g) \;=\; \sum_{x \in K} \Biggl( \sum_{y \in K} \mu(x,y)\, g(y) \Biggr) f(x) 
\;=\; \sum_{y \in K} \Biggl( \sum_{x \in K} \mu(x,y)\, f(x) \Biggr) g(y).
\end{equation}
Moreover, the expression in \eqref{Rel:01} is well-defined for all $f,g \in \ell_\infty(K)$, 
and therefore $G$ can be naturally extended to a bilinear form on $\ell_\infty(K)\times \ell_\infty(K)$ 
with the same norm. In addition, $\|G\| \;=\; \mathrm{Var}(\mu)(K,K)$.
\end{theorem}
\begin{proof}
By \cite[Chapter VI, Theorem~12]{Blei}, there exists a unique $\mu \in \mathcal{F}(K,K)$ such that  
\[
G(f,g) = \int f \otimes g \, d\mu = \int_K f(x)\, d\mu_g
\]
with $\|\mu\|_{\mathcal{F}_2} \leq \|G\|$. Since $K$ is a scattered compact space, $\mu_g$ is an atomic measure (see \cite[Definition 19.7.1]{Se}), and therefore
\[
G(f,g)=\sum_{x\in K} f(x)\mu_g(x),
\]
where this summation is well defined because $(\mu_g(x))_{x\in K}$ belongs to $\ell_1(K)$ (see \cite[Theorem~19.7.7]{Se}). Next, for each $x\in K$, since the marginal measure
$E\mapsto \mu(\{x\},E)$ is atomic, we have
\[
\mu_g(x)=\int_K g(y) \, d\mu(x, dy)=\sum_{y\in K} g(y)\mu(x,y),
\]
and again the summation is well defined since $(\mu(x,y))_{y\in K} \in \ell_1(K)$. Therefore,
\[
G(f,g) = \sum_{x\in K} \left( \sum_{y\in K} \mu(x,y)\, g(y) \right) f(x).
\]
The second equality in \eqref{Rel:01} follows from the Fubini-type relation mentioned earlier. 

Now let $f$ and $g$ be elements of the unit ball $B_{\ell_\infty(K)}$. Since $f(x),g(y)\in [-1,1]$ for every $x,y \in K$, it follows that for all finite sets $S,T\subset K$,
\[
\left|\sum_{x \in S} \sum_{y \in T}\mu(x,y) \, f(x) \, g(y) \right|
   \leq \mathrm{Var}(\mu)(K,K)\leq \|\mu\|_{\mathcal{F}_2}.
\]
This shows that \eqref{Rel:01} is well-defined and defines a bilinear form $\tilde{G}$ on $\ell_\infty(K)\times \ell_\infty(K)$ extending $G$. Moreover, 
we have $\|\tilde{G}\| \leq \mathrm{Var}(\mu)(K,K) \leq \|G\|$, and hence $\|G\| = \|\tilde{G}\| = \mathrm{Var}(\mu)(K,K)$.
\end{proof}

\begin{remark}\label{Rem:C_0(L)case}
In our analysis, we focus on bilinear forms on $C_0(L)$ rather than on $C(K)$. Recall that $C_0(L) = \{f \in C(K) : f(\infty)=0\}$; any bilinear form $G:C_0(L)\times C_0(L)\to \mathbb{R}$ admits a natural extension $\tilde{G}:C(K)\times C(K)\to \mathbb{R}$ defined by $\tilde{G}(f,g) = G(f-f(\infty), g-g(\infty))$. Applying Theorem~\ref{Thm:RieszBilinearScattered} to $\tilde{G}$ yields a unique $\mu \in \mathcal{F}(K,K)$ such that \eqref{Rel:01} holds when restricted to $C_0(L)$, and all remaining statements, including the extension to $\ell_\infty(K)\times \ell_\infty(K)$ and the norm equality, remain valid. Since $\mu(\{\infty\},U) = \mu(U,\{\infty\}) = 0$ for every $U\subset K$, hence $\mu \in \mathcal{F}_0(K,K)$.
\end{remark}

\begin{remark}
From Theorem \ref{Thm:RieszBilinearScattered} it follows that if $K$ is a scattered compact space and $\mu \in \mathcal{F}(K,K)$, then for every $U,V\subset K$ we have
\[
\mu(U,V) \;=\; \sum_{x \in U} \Biggl( \sum_{y \in V} \mu(x,y) \Biggr).
\]
This, in turn, yields the important fact that for every $U,V \in \wp(K)$ one has $|\mu(U,V)|\leq\mathrm{Var}(\mu)(U,V)$. Indeed, given $\epsilon>0$, there exist finite sets $S\subset U$ and $T\subset V$ such that 
\[
|\mu(U,V)| - \epsilon 
< \left| \sum_{x\in S}\sum_{y\in T} \mu(x,y)\right|
\;\leq\; \mathrm{Var}(\mu)(S,T)
\;\leq\; \mathrm{Var}(\mu)(U,V).
\]
\end{remark}

Finally, for every $\mu \in \mathcal{F}(K,K)$ we can establish a finite-support regularity property associated with $\mathrm{Var}(\mu)$ (Theorem~\ref{Thm:Regularity}), which will play a central role in the next section. To this end, we require the following auxiliary result, adapted from \cite[Chapter IV, Lemmas~3 and 4]{Blei}.

\begin{lemma}\label{Lem:RegularityAux1}
Let $K$ be a scattered compact Hausdorff space and let $\mu\in\mathcal{F}(K,K)$. 
Let $\{A_n\}_{n<\omega}$ and $\{B_n\}_{n<\omega}$ be sequences of pairwise disjoint sets such that $A_n\cap B_m=\varnothing$ whenever $m\neq n$. Then
\[\sum_{n=0}^\infty \mathrm{Var}(\mu)(A_n,B_n)\leq \mathrm{Var}(\mu)(K,K).\]
\end{lemma}

\begin{proof}
Let $\Omega=\{-1,1\}^\omega$ endowed with the usual product $\sigma$-algebra, and let $\mathbb{P}$ be the Bernoulli product measure on $\Omega$ (so that each coordinate is $\tfrac12\delta_{-1}+\tfrac12\delta_{1}$). Denote by $\{S_n: n<\omega\}$ the Rademacher system on $\Omega$.

For $h\in\Omega$ and $x\in K$, consider the functions $f_h,g_h\in \ell_\infty(K)$ given by 
\[
f_h(x)=\sum_{n=0}^\infty S_n(h)\,\chi_{A_n}(x), \qquad 
g_h(x)=\sum_{n=0}^\infty S_n(h)\,\chi_{B_n}(x).
\]
Note that $f_h$ and $g_h$ are well defined by the disjointness hypothesis on the collections $\{A_n\}_{n<\omega}$ and $\{B_n\}_{n<\omega}$.  
Moreover, if $\mathbb{E}$ denotes expectation with respect to $\mathbb{P}$, then for all $x,y\in K$,
\[\mathbb{E}\big[f_h(x)g_h(y)\big] \;=\; \sum_{n=0}^\infty \chi_{A_n\times B_n}(x,y).\]

Let $\epsilon>0$ be arbitrary. For each $i<\omega$ fix finite sets $S_i$ and $T_i$, together with coefficients $\{u_x:x\in S_i\}\subset [-1,1]$ and $\{v_y:y\in T_i\}\subset [-1,1]$, such that   
\[\mathrm{Var}(\mu)(A_i,B_i)-\frac{\epsilon}{2^i}<\sum_{x\in S_i}\sum_{y \in T_i}\mu(x,y)u_xv_y.\]

Next, given $n<\omega$, let $S=\bigcup_{i=0}^n S_i$ and $T=\bigcup_{i=0}^n T_i$. For each $h\in\Omega$, it is clear that $f_h,g_h\in B_{\ell_\infty(K)}$, hence by the definition of $\mathrm{Var}(\mu)(K,K)$,
\[
\sum_{x\in S}\sum_{y\in T} \mu(x,y)\, (f_h(x)u_x)\,(g_h(y)v_y) \le \mathrm{Var}(\mu)(K,K).
\]
Taking expectation with respect to $\mathbb{P}$ yields
\[
\sum_{i=0}^n\left(\sum_{x\in S_i}\sum_{y\in T_i} \mu(x,y)u_x v_y\right)=\sum_{x\in S}\sum_{y\in T} \mu(x,y)\,\mathbb{E}\big[f_h(x)g_h(y)\big] u_x v_y
\le \mathrm{Var}(\mu)(K,K).
\]
Therefore, for each $n<\omega$ we have
\[
\sum_{i=0}^n \mathrm{Var}(\mu)(A_i,B_i)\le \mathrm{Var}(\mu)(K,K)+\epsilon.
\]
Since $\epsilon>0$ is arbitrary, the conclusion follows.
\end{proof}

From Lemma \ref{Lem:RegularityAux1} we cal also deduce the following fact, which will be relevant at a later stage of this work.

\begin{proposition}\label{Prop:CountableSupportFrechetMeasures}
Let $K$ be a scattered compact Hausdorff space. If $\mu \in \mathcal{F}(K,K)$, then $\mu$ has countable support, that is, $\mathrm{Supp}(\mu) = \{(x,y)\in K\times K : \mu(x,y)\neq 0\}$ is countable.
\end{proposition}
\begin{proof}
Suppose, for a contradiction, that $\mathrm{Supp}(\mu)$ is uncountable. Then there exist $\epsilon>0$ and a sequence $\{(x_n,y_n)\}_{n<\omega}$ of pairwise distinct points in $\mathrm{Supp}(\mu)$ such that $|\mu(x_n,y_n)|\ge \epsilon$ for all $n<\omega$.

We define a coloring of the pairs $\{m,n\}$, $m<n$, into five colors as follows:
\[
\begin{array}{ll}
\text{Color 1:} & x_m = x_n \\
\text{Color 2:} & y_m = y_n \\
\text{Color 3:} & x_m = y_n \\
\text{Color 4:} & y_m = x_n \\
\text{Color 5:} & \text{none of the above equalities hold}
\end{array}
\]
By Ramsey's Theorem (infinite version), there exists an infinite set $H\subset \omega$ such that all pairs in $H$ have the same color. Write $H=\{n_k:k<\omega\}$ with $n_1<n_2<\ldots$

If the color were $1$, then all $x_{n_k}$ would coincide, say $x_{n_k}=x$ for every $k$, contradicting the fact that the marginal measure $E \mapsto \mu(\{x\},E)$ is finite. Similarly, if the color were $2$, all $y_{n_k}$ would coincide, leading to the same contradiction for the marginal measure $E \mapsto \mu(E,\{y\})$. The color cannot be $3$, since the points are pairwise distinct, and for the same reason, it cannot be $4$.

Therefore, the only possibility for $H$ is color $5$. Thus, the subsequence $\{(x_{n_k},y_{n_k})\}_{k<\omega}$ satisfies, for $k\neq j$,
\[
x_{n_k}\neq x_{n_j},\qquad y_{n_k}\neq y_{n_j},\qquad x_{n_k}\neq y_{n_j}.
\]
By Lemma \ref{Lem:RegularityAux1}, we then obtain
\[
\sum_{k=1}^\infty \mathrm{Var}(\mu)(\{x_{n_k}\},\{y_{n_k}\}) \le \mathrm{Var}(\mu)(K,K) <\infty,
\]
which is a contradiction since $\mathrm{Var}(\mu)(\{x_{n_k}\},\{y_{n_k}\})=|\mu(x_{n_k},y_{n_k})|\geq \epsilon$ for all $k$.
\end{proof}

\begin{lemma}\label{Lem:RegularityAux2}
Let $K$ be a scattered compact Hausdorff space and let $\mu \in \mathcal{F}(K,K)$. 
For every $\epsilon>0$, there exists a finite set $I_0 \subset K$ such that $\mathrm{Var}(\mu)(K\setminus I,K\setminus I)\leq \epsilon$ for every finite set $I \supset I_0$.
\end{lemma}
\begin{proof}
Suppose the proposition is false. Then there exists $\mu \in \mathcal{F}(K,K)$ and $\epsilon>0$ such that, for every finite set $I\subset K$, we have
$\mathrm{Var}(\mu)(K\setminus I,K\setminus I)> \epsilon$.

We construct recursively two sequences of finite sets $\{A_n\}_{n<\omega}$ and $\{B_n\}_{n<\omega}$ as follows. 
Let $A_0=B_0=\emptyset$, and suppose that we have already defined $\{A_0,\dots,A_n\}$ and $\{B_0,\dots,B_n\}$ as pairwise disjoint finite sets with $A_i\cap B_j=\emptyset$ whenever $i\neq j$. 
Set 
\[
I = \bigcup_{i=0}^n (A_i \cup B_i).
\]
Since $\mathrm{Var}(\mu)(K\setminus I,K\setminus I)\geq \epsilon$, there exist finite sets $S,T\subset K\setminus I$ and coefficients $\{u_x:x\in S\}\subset [-1,1]$, $\{v_y:y\in T\}\subset [-1,1]$ such that
\[
\epsilon < \sum_{x\in S}\sum_{y\in T} \mu(x,y) u_x v_y \le \mathrm{Var}(\mu)(K\setminus I,K\setminus I).
\]

We then set $A_{n+1}=S$ and $B_{n+1}=T$. By recursion, this constructs sequences $\{A_n\}_{n<\omega}$ and $\{B_n\}_{n<\omega}$ of pairwise disjoint finite sets with $A_i \cap B_j = \emptyset$ for $i\neq j$, and such that $\mathrm{Var}(\mu)(A_n,B_n)>\epsilon$ for all $n<\omega$. 

This contradicts Lemma \ref{Lem:RegularityAux1}, proving the proposition.
\end{proof}

\begin{theorem}\label{Thm:Regularity}
Let $K$ be a scattered compact Hausdorff space and let $\mu \in \mathcal{F}(K,K)$. For every $\epsilon>0$ and for every $a,b\in K$, there exists a finite set $I_0 \subset K$ such that, for every $I\subset K$ containing $I_0$, 
\[
\mathrm{Var}(\mu)(\{a\},K\setminus I)+\mathrm{Var}(\mu)(K\setminus I,\{b\})+\mathrm{Var}(\mu)(K\setminus I,K\setminus I)<\epsilon.
\]
\end{theorem}

\begin{proof}
Let $\epsilon>0$ be arbitrary. By Lemma \ref{Lem:RegularityAux2}, we may fix a finite set $J_0\subset K$ such that, for every  $J \supset J_0$,
\[
\mathrm{Var}(\mu)(K\setminus J,K\setminus J)\leq \frac{\epsilon}{3}.
\]

Given $a,b \in K$, let $\mu_a$ and $\mu^b$ denote the marginal measures in $\mathcal{M}(K)$ defined by $\mu_a(E)=\mu(\{a\},E)$ and $\mu^b(E)=\mu(E,\{b\})$ for every $E\in \mathfrak{B}(K)$, with total variation (in the usual sense of Radon measures) denoted by $|\mu_a|$ and $|\mu^b|$ respectively. Since $\mu_a$ and $\mu^b$ are atomic measures, there exist finite sets $J_1, J_2 \subset K$ such that, for every  $J\supset (J_1\cup J_2)$,
\[
\mathrm{Var}(\{a\},K\setminus J)=|\mu_a|(K\setminus J)<\frac{\epsilon}{3} \quad \text{and} \quad \mathrm{Var}(K\setminus J,\{b\})=|\mu^b|(K\setminus J)<\frac{\epsilon}{3}.
\]

The claim follows by taking $I_0 = J_0 \cup J_1 \cup J_2$.
\end{proof}

\section{Dismantling of Bilinear Operators}
\label{Sec:Decomposition}

In this section, we always consider $L$ to be a collapsing space (see Definition~\ref{Def:Collapsing}), introduced in Section~\ref{Sec:Basic}, and let $K = L \cup \{\infty\}$ denote its one-point compactification. Our aim is to establish, via a suitable characterization of bilinear operators $G:C_0(L)\times C_0(L)\to C_0(L)$, a series of results that will allow us to dismantle these operators by exploiting the exotic properties of $L$. We begin by introducing the main object of this section.

\begin{definition}\label{Def:SpiritFunction}
We say that a mapping $\nu:K \to \mathcal{F}_0(K,K)$, written for simplicity as $\nu(t) = \nu_t$, is a weak$^*$-continuous function vanishing at infinity if, for every $f,g \in C_0(L)$, the formula
\[
t \mapsto \int f\otimes g\, d\nu_t \;=\; \sum_{x \in K} \Biggl( \sum_{y \in K} \nu_t(x,y)\, g(y) \Biggr) f(x)
\]
defines an element of $C_0(L)$.
\end{definition}

This terminology is due to the fact that $\mathcal{F}_0(K,K)$, which is isometrically isomorphic to the space of bilinear forms on $C_0(L)\times C_0(L)$ by Theorem~\ref{Thm:RieszBilinearScattered}, is the topological dual of the projective tensor product $C_0(L)\,\hat{\otimes}_{\pi}\, C_0(L)$ (see \cite[Theorem~2.9]{Ryan}).  

The relevance of such functions in our research stems from the following characterization:

\begin{theorem}\label{Thm:Representation}
Every weak$^*$-continuous function $\nu:K \to \mathcal{F}_0(K,K)$ vanishing at infinity defines a bilinear operator $G : C_0(L) \times C_0(L) \longrightarrow C_0(L)$ through the formula
\begin{equation}\label{Rel:Formula}G(f,g)(t) \;=\; \sum_{x \in K} \Biggl( \sum_{y \in K} \nu_t(x,y)\, g(y) \Biggr) f(x) 
\quad\text{for all } f,g \in C_0(L), \; t \in K.
\end{equation}
Conversely, every bilinear operator $G:C_0(L)\times C_0(L)\to C_0(L)$ admits a unique weak$^*$-continuous function $\nu:K \to \mathcal{F}_0(K,K)$ satisfying the relation above. In either case, the operator norm is given by
\[\|G\| = \sup_{t \in K} \mathrm{Var}(\nu_t)(K,K).\]
\end{theorem}
\begin{proof}
If $\nu:K\to \mathcal{F}_0(K,K)$ is weak$^*$-continuous function vanishing at infinity, then the formula \ref{Rel:Formula} defines a bilinear mapping $G:C_0(L)\times C_0(L)\to C_0(L)$. According to the Banach-Steinhaus theorem (applied twice), we have $\|G\|<\infty$. For each $t \in K$, the mapping $G_t:C_0(L)\times C_0(L)\to \mathbb{R}$ given by $G_t(f,g)=G(f,g)(t)$ is a bilinear form on $C_0(L)$. By the uniqueness of the representation in Theorem \ref{Thm:RieszBilinearScattered}, we obtain $\|G_t\|=\mathrm{Var}(\nu_t)(K,K)$ and therefore
\[\|G\|=\sup_{t \in K}\|G_t\|=\sup_{t \in K}\mathrm{Var}(\nu_t)(K,K).\]

Conversely, if $G:C_0(L)\times C_0(L)\to C_0(L)$ is a bilinear operator, then $G_t(f,g)=G(f,g)(t)$ defines a bilinear form on $C_0(L)$. By Theorem \ref{Thm:RieszBilinearScattered} (see also Remark \ref{Rem:C_0(L)case}), there exists a unique $\nu_t\in\mathcal{F}_0(K,K)$ such that 
\[G(f,g)(t) \;=\; \sum_{x \in K} \Biggl( \sum_{y \in K} \nu_t(x,y)\, g(y) \Biggr) f(x).\] 
Since $G(f,g)\in C_0(L)$ for each $f,g\in C_0(L)$, it follows that $\nu:K\to \mathcal{F}_0(K,K)$ is weak$^*$-continuous function vanishing at infinity, as in Definition \ref{Def:SpiritFunction}. Moreover, by the second part of Theorem \ref{Thm:RieszBilinearScattered} we obtain 
\[\|G\|=\sup_{t \in K}\|G_t\|=\sup_{t \in K}\mathrm{Var}(\nu_t)(K,K).\]
\end{proof}

The following proposition was inspired by \cite[Theorem 4.4]{candido2}.

\begin{proposition}\label{Thm:Vanish-A} 
Let $A,B\subset L$ be countable subsets of $L$. If $\nu:K\to \mathcal{F}_0(K,K)$ denotes a weak$^*$-continuous function vanishing at infinity, then there exists $\rho<\omega_1$ such that $\mathrm{Var}(\nu_t)(A,B)=0$ for every $t\in K\setminus L_\rho$.
\end{proposition}
\begin{proof}
Suppose the proposition is false. Then, since $A$ and $B$ are countable and $L$ is uncountable, there exist $\epsilon>0$ and $(a,b)\in A\times B$ such that $|\nu_t(a,b)|>\epsilon$ for each $t$ in an uncountable set $\varLambda\subset L$. 
Since $L$ is a first-countable, locally compact, scattered space, we may fix countable local bases $\mathcal{U}_a=\{U_n : n<\omega\}$ and $\mathcal{V}_b=\{V_m : m <\omega\}$ consisting of compact clopen neighborhoods of $a$ and $b$, respectively.

For each $t \in \varLambda$, from the regularity of the measure $E\mapsto \nu_t(\{a\},E)$, we may select $V_t\in \mathcal{V}_b$ such that 
$|\nu_t(\{a\}, V_t\setminus \{b\})| < \frac{\epsilon}{4}$. Then by the regularity of the measure $E\mapsto \nu_t(E,V_t)$ we may select  $U_t\in \mathcal{U}_a$ such that $|\nu_t(U_t\setminus \{a\}, V_t)|<\frac{\epsilon}{4}$. Since $\mathcal{U}_a$ and $\mathcal{V}_b$ are countable, there exists an uncountable subset $\varLambda^\prime\subset \varLambda$ such that $U_t=U$ and $V_t=V$ for all $t\in \varLambda^\prime$.

For each $t \in \varLambda^\prime$ we have
\begin{align*}
|\nu_t(U,V)|&= |\nu_t(a,b) + \nu_t(U\setminus \{a\}, V) + \nu_t(\{a\}, V\setminus \{b\})| \\
&\geq |\nu_t(a,b)| - |\nu_t(U\setminus \{a\}, V)| - |\nu_t(\{a\}, V\setminus \{b\})| \\
&\geq \epsilon - \frac{\epsilon}{2} = \frac{\epsilon}{2}.
\end{align*}
This is a contradiction, since $t\mapsto \nu_t(U,V)=\int \chi_U\otimes \chi_V\, d\nu_t$ belongs to $C_0(L)$ and therefore must have countable support.
\end{proof}

\begin{proposition}\label{Prop:SomaSegmentada}
Let $(a_w)_{w\in L}$ be an element of $\ell_\infty(L)$ possessing the following property: for every $\rho<\omega_1$, $(a_w)_{w\in L_\rho}$ is an element of $\ell_1(L_\rho)$. Then, $(a_w)_{w\in L}$ belongs to $\ell_1(L)$.
\end{proposition}
\begin{proof}
Assume, towards a contradiction, that $(a_w)_{w\in L} \notin \ell_1(L)$. Then, for each $n<\omega$, there exists a finite set $F_n \subset L$ such that $\sum_{w\in F_n} |a_w| \geq n$. Let $\rho < \omega_1$ be such that $\bigcup_{n<\omega} F_n \subset L_\rho$. It follows that $\sup\Bigl\{ \sum_{w\in F} |a_w| : F \subset L_\rho \text{ finite} \Bigr\} = \infty$, which contradicts the assumption that $(a_w)_{w\in L_\rho} \in \ell_1(L_\rho)$.
\end{proof}

The next result was inspired by \cite[Lemma 5.9]{CandidoC(KxK)} and \cite[Theorem 4.8]{candido2}.

\begin{theorem}\label{Thm:Vanish-B}
Let $\nu:K\to \mathcal{F}_0(K\times K)$ be a weak$^*$-continuous function vanishing at infinity. Then there exist $(a_z)_{z},(b_z)_{z}\in \ell_1(L)$ such that, for each $z\in L$, there exists $\gamma_z<\omega_1$ with the property that, for all $x,y\in L\setminus L_{\gamma_z}$,
\[
\nu_{x}(z,y)=
\begin{cases}
a_z, & \text{if } x = y, \\
0,   & \text{if } x \neq y,
\end{cases}
\qquad
\nu_{x}(y,z)=
\begin{cases}
b_z, & \text{if } x = y, \\
0,   & \text{if } x \neq y.
\end{cases}
\]
\end{theorem}

\begin{proof}
Let $z \in L$ be arbitrary, and fix a sequence of compact clopen neighborhoods $V_1 \supset V_2 \supset V_3 \supset \ldots$ of $z$ in $L$ such that $\bigcap_{n < \omega} V_n = \{z\}$. For each $n < \omega$, define the functions 
$\nu^{(1,n)}, \nu^{(2,n)}: K \to \mathcal{M}_0(K)$ by
\begin{align*}
\nu^{(1,n)}_{x}(U) &= \nu_x(V_n,U),\\
\nu^{(2,n)}_{x}(U) &= \nu_x(U,V_n).
\end{align*}
These functions are weak$^*$-continuous and vanish at infinity (in the sense presented at Section \ref{Sec:Basic}). We analyze the sequence $\{\nu^{(1,n)}\}_{n<\omega}$ and the arguments for $\{\nu^{(2,n)}\}_{n<\omega}$ are analogous. Since $L$ is a collapsing space, for each $n < \omega$ Theorem \ref{Thm:AuxiliarCallapse2} yields $s_n \in \mathbb{R}$ and $\beta_n < \omega_1$ such that
\[\nu^{(1,n)}_{x}(y)= \begin{cases}
s_n, & \text{if } x = y, \\
0,   & \text{if } x \neq y,
\end{cases}\]
for all $x,y \in L \setminus L_{\beta_n}$. Define $\xi_1(z) = \sup\{beta_n:n < \omega\}$. Let $m, n < \omega$, pick $x \in L \setminus L_{\xi_1(z)}$ and let $\nu_x^x\in \mathcal{M}_0(K)$ be the marginal measure $\nu_x^x(U)=\nu_x(U,\{x\})$. If $n \leq m$, then
\begin{align*}
|s_n - s_m| &= |\nu^{(1,n)}_{x}(x) - \nu^{(1,m)}_{x}(x)|=|\nu_{x}(V_n,\{x\}) - \nu_{x}(V_m,\{x\})| \\
&= |\nu_x(V_n \setminus V_m,\{x\})|= |\nu_x^x(V_n \setminus V_m)| \leq |\nu_x^x|(V_n \setminus \{z\}).
\end{align*}
Since $\bigcap_{n<\omega} (V_n \setminus \{z\}) = \emptyset$, it follows that $\lim_{n \to \infty} |\nu_x^x|(V_n \setminus \{z\}) = 0$, so the sequence $\{s_n\}_{n<\omega}$ converges to some limit $a_z \in \mathbb{R}$. Hence, for each $x \in L \setminus L_{\xi_1(z)}$ and $n < \omega$,
\begin{align*}
|\nu_x(z,x) -a_z| &\leq |\nu_x(z,x) - s_n| +|s_n -a_z|    \\
&\leq |\nu_x(V_n \setminus \{z\}, \{x\})| + |s_n - a_z|\\
&\leq |\nu_x^x|(V_n\setminus \{z\}) + |s_n - a_z|,
\end{align*}
which implies $\nu_x(z,x) = a_z$. Now let $x, y \in L \setminus L_{\xi_1(z)}$ with $x \neq y$, and fix $n < \omega$. If $\nu_x^y\in \mathcal{M}_0(K)$ denotes the marginal measure $\nu_x^y(U)=\nu_x(U,\{y\})$, we have
\begin{align*}
|\nu_x(z,y)|&\leq |\nu^{(1,n)}_{x}(y)|+|\nu^{(1,n)}_{x}(y)-\nu_x(z,y)| \\
&=|\nu_{x}(V_n\setminus \{z\},\{y\})|\leq |\nu_x^y|(V_n \setminus \{z\}).
\end{align*}
Since $\bigcap_{n<\omega} (V_n \setminus \{z\}) = \emptyset$, it follows that $\nu_x(z,y) = 0$. 

Applying the same reasoning to $\{\nu^{(2,n)}\}_n$, we obtain an ordinal $\xi_2(z)$ and a real number $b_z$ such that, whenever $x,y \in L\setminus L_{\xi_2(z)}$, if $x=y$ then $\nu_x(y,z)=b_z$, and $\nu_x(y,z)=0$ if $x\neq y$. We set $\gamma_z=\max\{\xi_1(z),\xi_2(z)\}$. 

Finally, we prove that $(a_z)_{z}$ and $(b_z)_{z}$ belong to $\ell_1(L)$. Let $G:C_0(L)\times C_0(L)\to C_0(L)$ be the bilinear operator induced by $\nu$ as in Theorem \ref{Thm:Representation}. Let $\beta<\omega_1$ be arbitrary and set 
$\rho=\sup\{\gamma_z:z\in L_\beta\}$. Pick $t \in L\setminus L_\rho$ and consider the bilinear form $G_t:C_0(L)\times C_0(L)\to \mathbb{R}$ given by $G_t(f,g)=G(f,g)(t)$. If $\tilde{G}_t$ denotes the extension of $G_t$ to $\ell_\infty(L)\times \ell_\infty(L)$ as in Theorem \ref{Thm:RieszBilinearScattered}, then the formula
\begin{align*}
\varphi(f)=\tilde{G}_t(f,\chi_t)
= \sum_{z\in L_\beta} \left(\sum_{y\in K} \nu_t(z,y)\chi_t(y)\right)f(z)
= \sum_{z\in L_\beta}\nu_t(z,t) f(z)
= \sum_{z\in L_\beta}a_z f(z) 
\end{align*}
defines a bounded linear functional $\varphi:C_0(L_\rho)\to \mathbb{R}$. Therefore $(a_z)_{z\in L_\beta}$ belongs to $\ell_1(L_\rho)$ (see \cite[Corollary 19.7.7]{Se}). From Proposition \ref{Prop:SomaSegmentada} we deduce that $(a_z)_{z\in L}$ lies in $\ell_1(L)$, and the same reasoning applies to $(b_z)_{z\in L}$.
\end{proof}

For the next two lemmas, recalling that $L_\alpha=\varphi^{-1}[\alpha+1]$ as presented in Section \ref{Sec:Basic}, we denote $L_\alpha^*=\varphi^{-1}[\alpha]$.

\begin{lemma}\label{Lem:SeparationAnxiety}
Let $\nu:K\to \mathcal{F}_0(K,K)$ be weak$^*$-continuous function vanishing at infinity. Suppose that, for each $z\in L$, there exists $\gamma_z<\omega_1$ such that, for all $x,y\in L\setminus L_{\gamma_z}$, we have $\nu_x(y,z)=0$. Then, for every $\xi<\omega_1$, there exists $\xi\leq \rho<\omega_1$ such that, for every $t \in L\setminus L_\rho$ and every $(u,v)\in (L^*_\rho\times K\setminus L^*_\rho)\ \cup\ (K\setminus L^*_\rho \times L^*_\rho)\ \cup\ (L^*_\rho \times L^*_\rho)$, one has $\nu_t(u,v)=0$.
\end{lemma}

\begin{proof}
Fix $\xi<\omega_1$ and set $\rho_0=\xi$. Given $n<\omega$, suppose that $\rho_n$ has been defined. By Theorem \ref{Thm:Vanish-A}, there exists $\rho_n\leq \theta<\omega_1$ such that $\mathrm{Var}(\nu_y)(L_{\rho_n},L_{\rho_n})=0$ for every $y \in L\setminus L_\theta$. Moreover, by hypothesis, for each $z\in L_{\rho_n}$ there exists $\theta\leq \gamma_z<\omega_1$ such that $\nu_x(y,z)=0$ whenever $x,y\in L\setminus L_{\gamma_z}$. Define
\[
\rho_{n+1}=\sup\{\gamma_z: z\in L_{\rho_n}\}.
\]

By recursion, this yields a strictly increasing sequence $\rho_0<\rho_1<\rho_2<\ldots$ such that, for every $z\in L_{\rho_n}$ and all $x,y\in K\setminus L_{\rho_{n+1}}$, one has $\nu_x(y,z)=0$ and $\nu_x(z,z)=0$.  

Let $\rho=\sup\{\rho_n:n<\omega\}$. If $(u,v)\in L^*_\rho\times (K\setminus L^*_\rho)$ and $t\in L\setminus L_\rho$, then some $\rho_n<\rho$ satisfies $u\in L_{\rho_n}$, while $t,v\in K\setminus L^*_\rho\subset K\setminus L_{\rho_{n+1}}$, giving $\nu_t(u,v)=0$. A symmetric argument covers the case $(u,v)\in (K\setminus L^*_\rho)\times L^*_\rho$.  

Finally, if $(u,v)\in L^*_\rho\times L^*_\rho$, then $(u,v)\in L_{\rho_n}\times L_{\rho_n}$ for some $n$, and for all $t\in L\setminus L_\rho\subset L\setminus L_{\rho_{n+1}}$ we obtain $\nu_t(u,v)=0$. This completes the proof.
\end{proof}

\begin{lemma}\label{Lem:Separation}
Let $\nu:K\to \mathcal{F}_0(K,K)$ be a weak$^*$-continuous function vanishing at infinity. 
Suppose that, for each $z\in L$, there exists $\gamma_z<\omega_1$ such that, for all $x,y\in L\setminus L_{\gamma_z}$, we have $\nu_x(y,z)=0$. 
Let $\{x_\alpha\}_{\alpha<\omega_1}$, $\{y_\alpha\}_{\alpha<\omega_1}$, $\{z_\alpha\}_{\alpha<\omega_1}$ and $\{w_\alpha\}_{\alpha<\omega_1}$ be sequences of pairwise distinct points of $L$. 
Then, for every $\epsilon>0$, there exists a simultaneous refinement 
$\{x_{\alpha_\gamma}\}_{\gamma<\omega_1}$, $\{y_{\alpha_\gamma}\}_{\gamma<\omega_1}$, $\{z_{\alpha_\gamma}\}_{\gamma<\omega_1}$ and $\{w_{\alpha_\gamma}\}_{\gamma<\omega_1}$ of these sequences (see Definiton \ref{Def:Refinement}), together with a sequence $\{I_\gamma\}_{\gamma<\omega_1}$ of finite subsets of $L$ such that:
\begin{enumerate}[label=(B\alph*)]
\item\label{it:B1} $\{x_{\alpha_\gamma}, y_{\alpha_\gamma}, z_{\alpha_\gamma}, w_{\alpha_\gamma}\}\subset I_\gamma$ for all $\gamma<\omega_1$;   
\item\label{it:B2} $I_\gamma \cap I_\xi = \emptyset$ whenever $\gamma \neq \xi$;
\item\label{it:B3} $\mathrm{Var}(\nu_{x_{\alpha_\gamma}})(\{y_{\alpha_\gamma}\},K\setminus I_\gamma)
+\mathrm{Var}(\nu_{x_{\alpha_\gamma}})(K\setminus I_\gamma,\{z_{\alpha_\gamma}\})
+\mathrm{Var}(\nu_{x_{\alpha_\gamma}})(K\setminus I_\gamma,K\setminus I_\gamma)< \epsilon$;
\item\label{it:B4} $\mathrm{Var}(\nu_{x_{\alpha_\gamma}})(\{w_{\alpha_\gamma}\},K\setminus I_\gamma)+\mathrm{Var}(\nu_{x_{\alpha_\gamma}})(K\setminus I_\gamma,\{w_{\alpha_\gamma}\})<\epsilon$.
\end{enumerate}
Additionally, if $n<\omega$ is a fixed natural number, we may assume $|I_\gamma|\geq n$ for each $\gamma<\omega_1$.
\end{lemma}

\begin{proof}
We proceed by transfinite induction. Let $\gamma<\omega_1$ be arbitrary and assume that we have obtained sequences $\{x_{\alpha_\xi}\}_{\xi<\gamma}$, $\{y_{\alpha_\xi}\}_{\xi<\gamma}$, $\{z_{\alpha_\xi}\}_{\xi<\gamma}$, $\{w_{\alpha_\xi}\}_{\xi<\gamma}$, and a sequence $\{I_\xi\}_{\xi<\gamma}$ of finite subsets of $L$ satisfying conditions \ref{it:B1}–\ref{it:B4}, together with $|I_\xi|\geq n$ for each $\xi<\gamma$, where $n$ is a fixed natural number. Choose $\theta<\omega_1$ such that 
\[
\bigcup_{\delta<\xi} I_\xi \times I_\xi \subset L_{\theta}\times L_{\theta}.
\]

By Lemma \ref{Lem:SeparationAnxiety}, there exists $\theta\leq \rho<\omega_1$ such that, for each $t\in L\setminus L_{\rho}$ and every 
\[
(u,v)\in (L^*_{\rho}\times (K\setminus L^*_{\rho})) \cup ((K\setminus L^*_{\rho})\times L^*_{\rho}) \cup (L^*_{\rho}\times L^*_{\rho}),
\]
we have $\nu_t(u,v)=0$.

We now choose $\alpha_\gamma<\omega_1$, strictly greater than  $\sup\{\alpha_\xi:\xi<\gamma\}$, such that $x_{\alpha_\gamma},y_{\alpha_\gamma},z_{\alpha_\gamma},w_{\alpha_\gamma}\in L\setminus L_{\rho}$. According to Theorem \ref{Thm:Regularity}, we may obtain a finite set $J_0\subset L$ containing the points $x_{\alpha_\gamma},y_{\alpha_\gamma},z_{\alpha_\gamma},w_{\alpha_\gamma}$ such that for every $J\subset L$ containing $J_0$ we have
\[
\mathrm{Var}(\nu_{x_{\alpha_\gamma}})(\{y_{\alpha_\gamma}\},K\setminus J)
+\mathrm{Var}(\nu_{x_{\alpha_\gamma}})(K\setminus J,\{z_{\alpha_\gamma}\})
+\mathrm{Var}(\nu_{x_{\alpha_\gamma}})(K\setminus J,K\setminus J)< \epsilon.
\]

Next, we let $\mu_1$ and $\mu_2$ denote the marginal measures defined by $\mu_1(E)=\nu_{x_{\alpha_\gamma}}(\{w_{\alpha_\gamma}\},E)$ and $\mu_2(E)=\nu_{x_{\alpha_\gamma}}(E,\{w_{\alpha_\gamma}\})$ for every $E\in \mathfrak{B}(K)$, with total variation (in the usual sense of Radon measures) denoted $|\mu_1|$ and $|\mu_2|$, respectively. Since $\mu_1$ and $\mu_2$ are atomic measures, there exist finite sets $J_1, J_2 \subset K$ such that, for every $J\supset (J_1\cup J_2)$,
\[
\mathrm{Var}(\nu_{x_{\alpha_\gamma}})(\{w_{\alpha_\gamma}\},K\setminus J)=|\mu_1|(K\setminus J)<\tfrac{\epsilon}{2}, 
\quad 
\mathrm{Var}(\nu_{x_{\alpha_\gamma}})(K\setminus J,\{w_{\alpha_\gamma}\})=|\mu_2|(K\setminus J)<\tfrac{\epsilon}{2}.
\]

We then fix a finite set $J$ containing $J_0\cup J_1 \cup J_2$ such that $|J\setminus L_\xi|\geq n$. By the choice of $\rho$, we have $\nu_{x_{\alpha_\gamma}}(u,v)=0$ for each
\[
(u,v)\in (J\setminus L_{\xi}\times J\cap L_{\xi}) \cup (J\cap L_{\xi}\times J\setminus L_{\xi}) \cup (J\cap L_{\xi}\times J\cap L_{\xi}).
\]
Therefore, setting $I_\gamma=J\setminus L_{\xi}$, conditions \ref{it:B1}–\ref{it:B4} are satisfied for $\{x_{\alpha_\xi}\}_{\xi\leq\gamma}$, $\{y_{\alpha_\xi}\}_{\xi\leq\gamma}$, $\{z_{\alpha_\xi}\}_{\xi\leq \gamma}$, $\{w_{\alpha_\xi}\}_{\xi\leq \gamma}$, and $\{I_\xi\}_{\xi\leq \gamma}$. The proof then follows by transfinite induction.
\end{proof}

\begin{remark}\label{Rem:Equality}
In the previous lemma, the sequences 
$\{x_{\alpha}\}_{\alpha<\omega_1}$, $\{y_{\alpha}\}_{\alpha<\omega_1}$, 
$\{z_{\alpha}\}_{\alpha<\omega_1}$, and $\{w_{\alpha}\}_{\alpha<\omega_1}$ 
need not be pairwise distinct; in fact, two or more of them may coincide. 
We also emphasize that condition (4) of the lemma will only play a role 
in Theorem \ref{Thm:Vanish-E}.
\end{remark}

\begin{theorem}\label{Thm:Vanish-C}
Let $\nu:K\to \mathcal{F}_0(K,K)$ be weak$^*$-continuous vanishing at infinity. Suppose that, for each $z\in L$, there exists $\gamma_z<\omega_1$ such that, for all $x,y\in L\setminus L_{\gamma_z}$, we have $\nu_x(y,z)=0$. There exists $\rho<\omega$ such that, for all $x,y,z\in L\setminus L_{\rho}$, if $x\not\in \{y,z\}$ then one has $\nu_x(y,z)=0$. 
\end{theorem}
\begin{proof}
If the theorem were false, for some $\epsilon>0$ we could obtain transfinite sequences of pairwise distinct points 
$\{x_\alpha\}_{\alpha<\omega_1}$, $\{y_\alpha\}_{\alpha<\omega_1}$, and $\{z_\alpha\}_{\alpha<\omega_1}$ such that, 
for each $\alpha<\omega_1$, we have $x_\alpha \notin \{y_\alpha,z_\alpha\}$ and  $|\nu_{x_\alpha}(y_\alpha,z_\alpha)|\geq \epsilon$.

By employing Lemma \ref{Lem:Separation}, we may obtain a simultaneous refinement of these sequences, together with a sequence 
$\{I_\alpha\}_{\alpha<\omega_1}$ of finite subsets of $L$ of cardinality strictly greater than $3$, such that:
\begin{enumerate}[label=(C\alph*)]
\item\label{it:C1} $\{x_{\alpha}, y_{\alpha}, z_{\alpha}\}\subset I_\alpha$ for all $\alpha<\omega_1$;   
\item\label{it:C1} $I_\alpha \cap I_\beta = \emptyset$ whenever $\alpha \neq \beta$;
\item\label{it:C1} $\mathrm{Var}(\nu_{x_{\alpha}})(\{y_{\alpha}\},K\setminus I_\alpha)
+\mathrm{Var}(\nu_{x_{\alpha}})(K\setminus I_\alpha,\{z_{\alpha}\})
+\mathrm{Var}(\nu_{x_{\alpha}})(K\setminus I_\alpha,K\setminus I_\alpha)\leq \frac{\epsilon}{2}$.
\end{enumerate}

We may take another simultaneous refinement so that either $y_\alpha\neq z_\alpha$ for all $\alpha<\omega_1$ or $y_\alpha=z_\alpha$ for all $\alpha<\omega_1$. 
We proceed with the proof for the former case, and the latter can be obtained with simple adjustments.  
By taking a further simultaneous refinement, we may assume that $|I_\alpha|=m+3$ (with $m>0$) for each $\alpha<\omega_1$. 
Denoting $I_\alpha=\{y_\alpha,z_\alpha, x_\alpha, w_1^\alpha,\ldots,w_m^\alpha\}$, we form the collection 
\[
\mathcal{C} = \{(y_\alpha,z_\alpha,x_\alpha,w_1^\alpha, \ldots, w_{m}^\alpha) : \alpha < \omega_1\} \subset L^{\, m+3}
\]
and note that it forms an uncountable $(\{1\},\{2\},\{3,4,\ldots,m+3\})$-separated subset of $L^{\, m+3}$.

Since $L$ is a collapsing space, see Definition \ref{Def:Collapsing}, there exist $a,b\in L$, with $a\neq b$, such that 
$\ell = (a,b,\infty, \ldots, \infty)$ is an accumulation point of $\mathcal{C}$. 
We then fix a net $\{(y_{\alpha_\gamma},z_{\alpha_\gamma},x_{\alpha_\gamma},w_1^{\alpha_\gamma}, \ldots, w_{n}^{\alpha_\gamma})\}_{\gamma \in \Gamma}$  in  $\mathcal{C}$ converging to the point $\ell$.

Next, we select disjoint clopen compact neighborhoods $U$ of $a$ and $V$ of $b$, and, by passing to a subnet if necessary, we can assume that 
$U \cap I_{\alpha_{\gamma}} = \{y_{\alpha_{\gamma}}\}$ and $V \cap I_{\alpha_{\gamma}} = \{z_{\alpha_{\gamma}}\}$ for each $\gamma \in \Gamma$. 
Therefore, we have $U\setminus \{y_{\alpha_{\gamma}}\}\subset K\setminus I_{\alpha_{\gamma}}$ and 
$V\setminus \{z_{\alpha_{\gamma}}\}\subset K\setminus I_{\alpha_{\gamma}}$. Hence,

\begin{align*}
|\nu_{x_{\alpha_{\gamma}}}(U,V)| 
&\geq |\nu_{x_{\alpha_{\gamma}}}(y_{\alpha_{\gamma}},z_{\alpha_{\gamma}})| 
- \Big( 
    |\nu_{x_{\alpha_{\gamma}}}(\{y_{\alpha_{\gamma}}\},V\setminus \{z_{\alpha_{\gamma}}\})| \\
&\quad + |\nu_{x_{\alpha_{\gamma}}}(U\setminus \{y_{\alpha_{\gamma}}\},\{z_{\alpha_{\gamma}}\})| 
    + |\nu_{x_{\alpha_{\gamma}}}(U\setminus \{y_{\alpha_{\gamma}}\},V\setminus \{z_{\alpha_{\gamma}}\})| 
  \Big)\\
&\geq \epsilon - \Big(
    \mathrm{Var}(\nu_{x_{\alpha_{\gamma}}})(\{y_{\alpha_{\gamma}}\}, V\setminus \{z_{\alpha_{\gamma}}\})\\
&\quad + \mathrm{Var}(\nu_{x_{\alpha_{\gamma}}})(U\setminus \{y_{\alpha_{\gamma}}\},\{z_{\alpha_{\gamma}}\})  
    + \mathrm{Var}(\nu_{x_{\alpha_{\gamma}}})(U\setminus \{y_{\alpha_{\gamma}}\}, V\setminus \{z_{\alpha_{\gamma}}\})
  \Big)\\
&\geq \epsilon - \Big(
    \mathrm{Var}(\nu_{x_{\alpha_\gamma}})(\{y_{\alpha_{\gamma}}\}, K\setminus I_{\alpha_\gamma}) \\
&\quad + \mathrm{Var}(\nu_{x_{\alpha_\gamma}})(K\setminus I_{\alpha_\gamma},\{z_{\alpha_{\gamma}}\}) 
    + \mathrm{Var}(\nu_{x_{\alpha_{\gamma}}})(K\setminus I_{\alpha_\gamma},K\setminus I_{\alpha_\gamma})
  \Big)\\
&> \epsilon - \frac{\epsilon}{2} = \frac{\epsilon}{2}.
\end{align*}

From the weak$^*$-continuity of the function $\nu$ and since $x_{\alpha_\gamma}\to \infty$, we obtain that $|\nu_\infty(U,V)|\geq \frac{\epsilon}{2}$, but this is a contradiction since $\nu_\infty=0$.
\end{proof}

\begin{theorem}\label{Thm:Vanish-D}Let $\nu:K\to \mathcal{F}_0(K,K)$ be weak$^*$-continuous vanishing at infinity. Suppose that, for each $z\in L$, there exists $\gamma_z<\omega_1$ such that, for all $x,y\in L\setminus L_{\gamma_z}$, we have $\nu_x(y,z)=0$. Then, there is $r\in \mathbb{R}$ and $\rho<\omega_1$ such that $\nu_y(y,y)=r$ for all $y\in L\setminus L_{\rho_1}$.
\end{theorem}
\begin{proof}
Towards a contradiction, assume that for every $r \in \mathbb{R}$ and every $\beta<\omega_1$, there exists $y \in L \setminus L_\beta$ such that $\nu_{y}(y,y)\neq r$. By \cite[Lemma 4.6]{candido2}, this yields rational numbers $p<q$ and uncountable sets $A,B\subset L$ such that 
\[
\nu_{y}(y,y)<p<q<\nu_{z}(z,z) \quad \text{for all } y\in A,\ z\in B.
\]

Since $A$ and $B$ are uncountable, we may construct sequences $\{y_\alpha\}_{\alpha<\omega_1}\subset A$ and $\{z_\alpha\}_{\alpha<\omega_1}\subset B$ of pairwise distinct points. Then, applying Lemma \ref{Lem:Separation} individually to each of these sequences (see Remark \ref{Rem:Equality}), we obtain subsequences $\{y_\alpha\}_{\alpha<\omega_1}\subset A$ and $\{z_\alpha\}_{\alpha<\omega_1}\subset B$ of pairwise distinct points, together with sequences 
$\{I_\alpha\}_{\alpha<\omega_1}$ and $\{J_\alpha\}_{\alpha<\omega_1}$ of finite subsets of $L$ with cardinality strictly greater than $2$, such that:
\begin{enumerate}[label=(D\alph*)]
\item\label{it:D1} $y_\alpha\in I_\alpha$ and $z_\alpha\in J_\alpha$ for all $\alpha<\omega_1$;   
\item\label{it:D2} $I_\alpha\cap I_\beta=\emptyset$ and $J_\alpha\cap J_\beta=\emptyset$ whenever $\alpha\neq\beta$;
\item\label{it:D3} for all $\alpha<\omega_1$,
\[
\mathrm{Var}(\nu_{y_{\alpha}})(\{y_{\alpha}\},K\setminus I_\alpha)+\mathrm{Var}(\nu_{y_{\alpha}})(K\setminus I_\alpha,\{y_{\alpha}\})+\mathrm{Var}(\nu_{y_{\alpha}})(K\setminus I_\alpha,K\setminus I_\alpha)\leq \tfrac{q-p}{3},
\]
and
\[
\mathrm{Var}(\nu_{z_{\alpha}})(\{z_{\alpha}\},K\setminus J_\alpha)+\mathrm{Var}(\nu_{z_{\alpha}})(K\setminus J_\alpha,\{z_{\alpha}\})+\mathrm{Var}(\nu_{z_{\alpha}})(K\setminus J_\alpha,K\setminus J_\alpha)\leq \tfrac{q-p}{3}.
\]
\end{enumerate}

Moreover, after another simultaneous refinement, we may assume that $\{I_\alpha \cup J_\alpha\}_{\alpha<\omega_1}$ is a sequence of pairwise disjoint sets satisfying $|I_\alpha \cup J_\alpha| = n+2$ (for some $n>0$) for every $\alpha < \omega_1$. Indeed, by applying the $\varDelta$-System Lemma, we may assume that $\{I_\alpha \cup J_\alpha\}_{\alpha<\omega_1}$ forms a $\varDelta$-system with root $R$, and that all members of the sequence have the same cardinality, strictly greater than $2$. Then, condition~\ref{it:D2} forces $R = \emptyset$.

Let us assume that $I_\alpha \cup J_\alpha = \{y_\alpha, z_\alpha, w_1^\alpha, \ldots, w_n^\alpha\}$ for all $\alpha < \omega_1$. This allows us to form the collection $\mathcal{C} = \{(y_\alpha, z_\alpha, w_1^\alpha, \ldots, w_n^\alpha) : \alpha < \omega_1\} \subset L^{n+2}$, which constitutes an uncountable $(\{1,2\}, \{3,4,\ldots,n+2\})$-separated set in $L^{n+2}$.

Since $L$ has the collapsing property, see Definition \ref{Def:Collapsing}, there exists a point $a \in L$ such that $\ell = (a,a,\infty,\ldots,\infty)$ is an accumulation point of $\mathcal{C}$.  We then fix a net $\{(y_{\alpha_\gamma},z_{\alpha_\gamma},w_1^{\alpha_\gamma}, \ldots, w_{n}^{\alpha_\gamma})\}_{\gamma \in \Gamma}$  in  $\mathcal{C}$ converging to the point $\ell$.  

Finally, we fix a clopen neighborhood $V$ of $a$ and, if necessary, by passing to a 
subnet we may assume that $V \cap I_{\alpha_\gamma} = \{y_{\alpha_\gamma}\}$ and $V \cap J_{\alpha_\gamma} = \{z_{\alpha_\gamma}\}$, for all $\gamma \in \varGamma$. It then follows that
\[
V \setminus \{y_{\alpha_\gamma}\} \subset K \setminus I_{\alpha_\gamma}
\quad \text{and} \quad 
V \setminus \{z_{\alpha_\gamma}\} \subset K \setminus J_{\alpha_\gamma},
\]
which in turn allows us to write:
\begin{align*}
\nu_{y_{\alpha_{\gamma}}}(V,V) 
&\leq \nu_{y_{\alpha_{\gamma}}}(y_{\alpha_{\gamma}},y_{\alpha_{\gamma}}) 
  + \Big(
      |\nu_{y_{\alpha_{\gamma}}}(\{y_{\alpha_{\gamma}}\}, V\setminus \{y_{\alpha_{\gamma}}\})| \\
&\quad + |\nu_{y_{\alpha_{\gamma}}}(V\setminus \{y_{\alpha_{\gamma}}\},\{y_{\alpha_{\gamma}}\})|
      + |\nu_{y_{\alpha_{\gamma}}}(V\setminus \{y_{\alpha_{\gamma}}\},V\setminus \{y_{\alpha_{\gamma}}\})|
    \Big)\\[0.5em]
&\leq p + \Big(
      \mathrm{Var}(\nu_{y_{\alpha_{\gamma}}})(\{y_{\alpha_{\gamma}}\},V\setminus \{y_{\alpha_{\gamma}}\}) \\
&\quad + \mathrm{Var}(\nu_{y_{\alpha_{\gamma}}})(V\setminus \{y_{\alpha_{\gamma}}\},\{y_{\alpha_{\gamma}}\}) 
      + \mathrm{Var}(\nu_{y_{\alpha_{\gamma}}})(V\setminus \{y_{\alpha_{\gamma}}\},V\setminus \{y_{\alpha_{\gamma}}\})
    \Big)\\[0.5em]
&\leq p + \Big(
      \mathrm{Var}(\nu_{y_{\alpha_{\gamma}}})(\{y_{\alpha_{\gamma}}\},K\setminus I_{\alpha_{\gamma}}) \\
&\quad + \mathrm{Var}(\nu_{y_{\alpha_{\gamma}}})(K\setminus I_{\alpha_{\gamma}},\{y_{\alpha_{\gamma}}\}) 
      + \mathrm{Var}(\nu_{y_{\alpha_{\gamma}}})(K\setminus I_{\alpha_{\gamma}},K\setminus I_{\alpha_{\gamma}})
    \Big)\\[0.5em]
&< p + \frac{q-p}{3} 
= \frac{2p+q}{3}.
\end{align*}

\begin{align*}
\nu_{z_{\alpha_{\gamma}}}(V,V) 
&\geq \nu_{z_{\alpha_{\gamma}}}(z_{\alpha_{\gamma}},z_{\alpha_{\gamma}}) 
   - \Big(
       |\nu_{z_{\alpha_{\gamma}}}(\{z_{\alpha_{\gamma}}\},V\setminus \{z_{\alpha_{\gamma}}\})| \\
&\quad + |\nu_{z_{\alpha_{\gamma}}}(V\setminus \{z_{\alpha_{\gamma}}\},\{z_{\alpha_{\gamma}}\})|
       + |\nu_{z_{\alpha_{\gamma}}}(V\setminus \{z_{\alpha_{\gamma}}\},V\setminus \{z_{\alpha_{\gamma}}\})|
     \Big)\\[0.5em]
&\geq q - \Big(
       \mathrm{Var}(\nu_{z_{\alpha_{\gamma}}})(\{z_{\alpha_{\gamma}}\},V\setminus \{z_{\alpha_{\gamma}}\}) \\
&\quad + \mathrm{Var}(\nu_{z_{\alpha_{\gamma}}})(V\setminus \{z_{\alpha_{\gamma}}\},\{z_{\alpha_{\gamma}}\}) 
       + \mathrm{Var}(\nu_{z_{\alpha_{\gamma}}})(V\setminus \{z_{\alpha_{\gamma}}\},V\setminus \{z_{\alpha_{\gamma}}\})
     \Big)\\[0.5em]
&\geq q - \Big(
       \mathrm{Var}(\nu_{z_{\alpha_{\gamma}}})(\{z_{\alpha_{\gamma}}\},K\setminus J_{\alpha_{\gamma}}) \\
&\quad + \mathrm{Var}(\nu_{z_{\alpha_{\gamma}}})(K\setminus J_{\alpha_{\gamma}},\{z_{\alpha_{\gamma}}\}) 
       + \mathrm{Var}(\nu_{z_{\alpha_{\gamma}}})(K\setminus J_{\alpha_{\gamma}},K\setminus J_{\alpha_{\gamma}})
     \Big)\\[0.5em]
&> q - \frac{q-p}{3} 
= \frac{2q+p}{3}.
\end{align*}

Recalling that $\nu$ is weak$^*$-continuous and that $V$ is a clopen compact subset of $L$, 
and since both nets $\{y_{\alpha_\gamma}\}_{\gamma\in \varGamma}$ and $\{z_{\alpha_\gamma}\}_{\gamma\in \varGamma}$ converge to $a$, 
the preceding inequalities yield
\[
\nu_a(V,V) 
= \lim_{\gamma\to \infty}\nu_{y_{\alpha_\gamma}}(V,V) 
   \leq \tfrac{2p+q}{3} 
   < \tfrac{2q+p}{3} 
   \leq \lim_{\gamma \to \infty}\nu_{z_{\alpha_\gamma}}(V,V) 
   = \nu_a(V,V),
\]
which is a contradiction.
\end{proof}

\begin{lemma}\label{Lem:SeparationPlus}
Let $\nu:K\to \mathcal{F}_0(K,K)$ be a weak$^*$-continuous function vanishing at infinity. Suppose that, for each $z\in L$, there exists $\gamma_z<\omega_1$ such that, for all $x,y\in L\setminus L_{\gamma_z}$, we have $\nu_x(y,z)=0$. Let $\{y_\alpha\}_{\alpha<\omega_1}$ and $\{z_\alpha\}_{\alpha<\omega_1}$ be sequences of pairwise distinct points of $L$ such that $y_\alpha\neq z_\alpha$ for each $\alpha<\omega_1$. Then there exist a simultaneous refinement $\{y_{\alpha_\gamma}\}_{\gamma<\omega_1}$ and $\{z_{\alpha_\gamma}\}_{\gamma<\omega_1}$ together with a new sequence $\{x_\gamma\}_{\gamma<\omega_1}$ of pairwise distinct points of $L$ such that, for all $\gamma,\xi,\eta<\omega_1$, the points $x_\gamma,y_{\alpha_\xi},z_{\alpha_\eta}$ are pairwise distinct and
\begin{enumerate}[label=(E\alph*)]
\item\label{it:E1} 
\[
\nu_{y_{\alpha_\gamma}}(x_\gamma,z_{\alpha_\gamma})=\nu_{y_{\alpha_\gamma}}(z_{\alpha_\gamma},x_\gamma)
=\nu_{z_{\alpha_\gamma}}(x_\gamma,y_{\alpha_\gamma})=\nu_{z_{\alpha_\gamma}}(y_{\alpha_\gamma},x_\gamma)=0,
\]
\item\label{it:E2} 
\[
\nu_{x_\gamma}(y_{\alpha_\gamma},x_\gamma)=\nu_{x_\gamma}(x_\gamma,y_{\alpha_\gamma})
=\nu_{x_\gamma}(z_{\alpha_\gamma},x_\gamma)=\nu_{x_\gamma}(x_\gamma,z_{\alpha_\gamma})=0.
\]
\end{enumerate}
\end{lemma}
\begin{proof}
Fix $\gamma<\omega_1$ and assume that we have obtained subsequences $\{y_{\alpha_\xi}\}_{\xi<\gamma}$ and $\{z_{\alpha_\xi}\}_{\xi<\gamma}$, and also a sequence $\{x_\xi\}_{\xi<\gamma}$, satisfying conditions \ref{it:E1} and \ref{it:E2} for each $\xi<\gamma$.

Choose $\alpha_\gamma$ strictly greater than $\sup\{\alpha_\xi:\xi<\gamma\}$ and let $\rho_1<\omega_1$ be such that
\[
\{y_{\alpha_\xi}:\xi\leq \gamma\}\cup \{z_{\alpha_\xi}:\xi\leq \gamma\}\cup \{x_\xi:\xi<\gamma\}\subset L_{\rho_1}.
\]
According to Proposition \ref{Prop:CountableSupportFrechetMeasures}, we may fix $\rho_2<\omega_1$ such that 
\[
\bigcup_{\xi\leq \gamma}\Bigl(\mathrm{Supp}(\nu_{y_{\alpha_\xi}})\cup \mathrm{Supp}(\nu_{z_{\alpha_\xi}})\Bigr)\subset L_{\rho_2}\times L_{\rho_2}.
\]
It follows that, for any $x\in L\setminus L_{\rho_2}$,
\[
\nu_{y_{\alpha_\xi}}(x,z_{\alpha_\xi})=\nu_{y_{\alpha_\xi}}(z_{\alpha_\xi},x)
=\nu_{z_{\alpha_\xi}}(x,y_{\alpha_\xi})=\nu_{z_{\alpha_\xi}}(y_{\alpha_\xi},x)=0
\]
for each $\xi\leq \gamma$. Moreover, by our hypothesis, there exists $\rho_3<\omega_1$ such that, for every $x\in L\setminus L_{\rho_3}$,
\[
\nu_x(y_{\alpha_\gamma},x)=\nu_x(x,y_{\alpha_\gamma})
=\nu_x(z_{\alpha_\gamma},x)=\nu_x(x,z_{\alpha_\gamma})=0.
\]
We then choose $x_\gamma\in L\setminus (L_{\rho_1}\cup L_{\rho_2}\cup L_{\rho_3})$. By construction, $x_\gamma,y_{\alpha_\gamma},z_{\alpha_\gamma}$ are pairwise distinct, and \ref{it:E1} and \ref{it:E2} are satisfied. This completes the inductive step. 

Hence, by transfinite recursion, the sequences $\{y_{\alpha_\gamma}\}_{\gamma<\omega_1}$, $\{z_{\alpha_\gamma}\}_{\gamma<\omega_1}$, and $\{x_\gamma\}_{\gamma<\omega_1}$ can be constructed as required.
\end{proof}

\begin{theorem}\label{Thm:Vanish-E}
Let $\nu:K\to \mathcal{F}_0(K,K)$ be a weak$^*$-continuous function vanishing at infinity. Suppose that, for each $z\in L$, there exists $\gamma_z<\omega_1$ such that, for all $x,y\in L\setminus L_{\gamma_z}$, we have $\nu_x(y,z)=0$. Let $\{y_\alpha\}_{\alpha<\omega_1}$ and $\{z_\alpha\}_{\alpha<\omega_1}$ be sequences of pairwise distinct points in $L$ such that $y_\alpha\neq z_\alpha$ for each $\alpha<\omega_1$. Then, there exists $\rho<\omega_1$ such that, for each $\alpha\geq \rho$, we have
\[
\nu_{y_\alpha}(y_\alpha,z_\alpha)=\nu_{y_\alpha}(z_\alpha,y_\alpha)=0.
\]
\end{theorem}
\begin{proof} If the theorem were false, by taking a simultaneous refinement of the sequences, we could find $\epsilon>0$ such that either $|\nu_{y_\alpha}(y_\alpha,z_\alpha)|\geq \epsilon$ for all $\alpha<\omega_1$, or $|\nu_{y_\alpha}(z_\alpha,y_\alpha)|\geq \epsilon$ for all $\alpha<\omega_1$. We assume, without loss of generality, the former situation; the latter can be treated with a similar argument with only minor adjustments.

By employing Lemma \ref{Lem:SeparationPlus}, we can find another simultaneous refinement, together with a sequence of pairwise distinct points $\{x_\alpha\}_{\alpha<\omega_1}$ such that $x_\alpha, y_\beta, z_\gamma$ are pairwise distinct for all $\alpha, \beta, \gamma < \omega_1$, and
\begin{align*}
\nu_{y_\alpha}(x_\alpha,z_\alpha)&=\nu_{y_\alpha}(z_\alpha,x_\alpha)=\nu_{z_\alpha}(x_\alpha,y_\alpha)=\nu_{z_\alpha}(y_\alpha,x_\alpha)=0,\\
\nu_{x_\alpha}(y_\alpha,x_\alpha)&=\nu_{x_\alpha}(x_\alpha,y_\alpha)=\nu_{x_\alpha}(z_\alpha,x_\alpha)=\nu_{x_\alpha}(x_\alpha,z_\alpha)=0
\end{align*}
for each $\alpha<\omega_1$.  

Next, by applying Lemma \ref{Lem:Separation} two times, we may obtain a simultaneous refinement of the above sequences, along with sequences $\{I_\alpha\}_{\alpha < \omega_1}$ and $\{J_\alpha\}_{\alpha < \omega_1}$ of finite subsets of $L$ with cardinality strictly grater than $3$, satisfying:
\begin{enumerate}[label=(F\alph*)]
\item\label{it:F1}  $\{x_\alpha, y_\alpha,z_\alpha\}\subset I_\alpha\cap J_\alpha$ for all $\alpha<\omega_1$;   
\item\label{it:F2} $I_\alpha\cap I_\beta=\emptyset$ and $J_\alpha\cap J_\beta=\emptyset$ whenever $\alpha\neq\beta$;
\item\label{it:F3} for all $\alpha<\omega_1$:
\[\mathrm{Var}(\nu_{x_{\alpha}})(\{y_{\alpha}\},K\setminus I_\alpha)+\mathrm{Var}(\nu_{x_{\alpha}})(K\setminus I_\alpha,\{z_{\alpha}\}) + \mathrm{Var}(\nu_{x_{\alpha}})(K\setminus I_\alpha,K\setminus I_\alpha) \leq \frac{\epsilon}{6},\]
and
\[\mathrm{Var}(\nu_{y_{\alpha}})(\{y_{\alpha}\},K\setminus J_\alpha)+\mathrm{Var}(\nu_{y_{\alpha}})(K\setminus J_\alpha,\{z_{\alpha}\}) + \mathrm{Var}(\nu_{y_{\alpha}})(K\setminus J_\alpha,K\setminus J_\alpha) \leq \frac{\epsilon}{6};\]
\item\label{it:F4}  for all $\alpha<\omega_1$:
\[\max\left\{\mathrm{Var}(\nu_{x_{\alpha}})(\{x_{\alpha}\},K\setminus I_\alpha), \mathrm{Var}(\nu_{y_{\alpha}})(\{x_{\alpha}\},K\setminus I_\alpha)\right\}<\frac{\epsilon}{6}.\]
\end{enumerate}

Note that we may assume that $\{I_\alpha \cup J_\alpha\}_{\alpha<\omega_1}$ is a sequence of pairwise disjoint sets with $|I_\alpha \cup J_\alpha| = n+3$ (with $n>0$) for every $\alpha<\omega_1$. Indeed, by refining the sequence $\{I_\alpha \cup J_\alpha\}_{\alpha<\omega_1}$ via the $\varDelta$-System Lemma, we may suppose that $\{I_\alpha \cup J_\alpha\}_{\alpha<\omega_1}$ forms a $\varDelta$-system with root $R$, and that all members of the sequence have the same cardinality strictly greater than $3$. However, the Condition \ref{it:F2} forces $R = \emptyset$.

We suppose that $I_\alpha \cup J_\alpha = \{x_\alpha, y_\alpha, z_\alpha, w_1^\alpha, \ldots, w_n^\alpha\}$,
and then form the collection \[\mathcal{C} = \{(x_\alpha, y_\alpha, z_\alpha, w_1^\alpha, \ldots, w_n^\alpha) : \alpha < \omega_1\} \subset L^{\,n+3},\]
noting that it constitutes an uncountable $(\{1,2\}, \{3\},\{4, \ldots, n+2\})$-separated subset of $L^{\,n+3}$.

Since $L$ is a collapsing space, see Definition \ref{Def:Collapsing}, there exist distinct points $a,b \in L$ with $a \neq b$ such that $\ell = (a,a,b,\infty, \ldots, \infty)$ is an accumulation point of $\mathcal{C}$. We let $\{(x_{\alpha_\gamma}, y_{\alpha_\gamma}, z_{\alpha_\gamma}, w_1^{\alpha_\gamma}, \ldots, w_n^{\alpha_\gamma})\}_{\gamma \in \Gamma}$ be a net in $\mathcal{C}$ converging to the point $\ell$.

Next, we choose disjoint clopen neighborhoods $U$ of $a$ and $V$ of $b$, and, by passing to a subnet if necessary, we may assume that 
\[
U \cap I_{\alpha_\gamma} = U \cap J_{\alpha_\gamma} = \{x_{\alpha_\gamma}, y_{\alpha_\gamma}\}
\quad\text{and}\quad
V \cap I_{\alpha_\gamma} = V \cap J_{\alpha_\gamma} = \{z_{\alpha_\gamma}\}
\]
for each $\gamma \in \Gamma$.

From the Condition \ref{it:F4}, we have for each $\gamma\in \varGamma$:
\begin{align*}
\bigl|\nu_{x_{\alpha_\gamma}}(U \setminus \{y_{\alpha_\gamma}\},\, V \setminus \{z_{\alpha_\gamma}\})\bigr|
&\leq \mathrm{Var}(\nu_{x_{\alpha_\gamma}})(U \setminus \{y_{\alpha_\gamma}\},\, V \setminus \{z_{\alpha_\gamma}\}) \\[0.3em]
&\leq \mathrm{Var}(\nu_{x_{\alpha_\gamma}})\bigl((K \setminus I_{\alpha_\gamma}) \cup \{x_{\alpha_\gamma}\},\, K \setminus I_{\alpha_\gamma}\bigr) \\[0.3em]
&\leq \mathrm{Var}(\nu_{x_{\alpha_\gamma}})(K \setminus I_{\alpha_\gamma},\, K \setminus I_{\alpha_\gamma}) 
   + \mathrm{Var}(\nu_{x_{\alpha_\gamma}})(\{x_{\alpha_\gamma}\},\, K \setminus I_{\alpha_\gamma}) \\[0.3em]
&\leq \mathrm{Var}(\nu_{x_{\alpha_\gamma}})(K \setminus I_{\alpha_\gamma},\, K \setminus I_{\alpha_\gamma}) + \frac{\epsilon}{6}.
\end{align*}
And since $\nu_{x_{\alpha_{\gamma}}}(x_{\alpha_{\gamma}},z_{\alpha_{\gamma}})=0$ we have:
\begin{align*}\mathrm{Var}(\nu_{x_{\alpha_\gamma}})(U\setminus \{y_{\alpha_\gamma}\},\{z_{\alpha_\gamma}\})&\leq \mathrm{Var}(\nu_{x_{\alpha_\gamma}})((K\setminus I_{\alpha_\gamma})\cup \{x_{\alpha_\gamma}\},\{z_{\alpha_\gamma}\})\\
&\leq \mathrm{Var}(\nu_{x_{\alpha_\gamma}})(K\setminus I_{\alpha_\gamma},\{z_{\alpha_\gamma}\})+|\nu_{x_{\alpha_\gamma}}(x_{\alpha_\gamma},z_{\alpha_\gamma})|\\
&=\mathrm{Var}(\nu_{x_{\alpha_\gamma}})(K\setminus I_{\alpha_\gamma}, \{z_{\alpha_\gamma}\})
\end{align*}

Therefore, recalling that $\nu_{x_{\alpha_\gamma}}(y_{\alpha_\gamma},z_{\alpha_\gamma})=0$ and Condition \ref{it:F3}, we can write
\begin{align*} 
|\nu_{x_{\alpha_\gamma}}(U,V)|
&\leq |\nu_{x_{\alpha_\gamma}}(y_{\alpha_\gamma},z_{\alpha_\gamma})|
   + \bigl|\nu_{x_{\alpha_\gamma}}(\{y_{\alpha_\gamma}\},\, V \setminus \{z_{\alpha_\gamma}\})\bigr| \\
&\quad + \bigl|\nu_{x_{\alpha_\gamma}}(U \setminus \{y_{\alpha_\gamma}\},\, \{z_{\alpha_\gamma}\})\bigr|
   + \bigl|\nu_{x_{\alpha_\gamma}}(U \setminus \{y_{\alpha_\gamma}\},\, V \setminus \{z_{\alpha_\gamma}\})\bigr| \\[0.5em]
&\leq \mathrm{Var}(\nu_{x_{\alpha_\gamma}})(\{y_{\alpha_\gamma}\},\, V \setminus \{z_{\alpha_\gamma}\}) \\
&\quad + \mathrm{Var}(\nu_{x_{\alpha_\gamma}})(U \setminus \{y_{\alpha_\gamma}\},\, \{z_{\alpha_\gamma}\})
   + \mathrm{Var}(\nu_{x_{\alpha_\gamma}})(U \setminus \{y_{\alpha_\gamma}\},\, V \setminus \{z_{\alpha_\gamma}\}) \\[0.5em]
&\leq \mathrm{Var}(\nu_{x_{\alpha_\gamma}})(\{y_{\alpha_\gamma}\},\, K \setminus I_{\alpha_\gamma}) \\
&\quad + \mathrm{Var}(\nu_{x_{\alpha_\gamma}})(K \setminus I_{\alpha_\gamma},\, \{z_{\alpha_\gamma}\}) + \mathrm{Var}(\nu_{x_{\alpha_\gamma}})(K \setminus I_{\alpha_\gamma},\, K \setminus I_{\alpha_\gamma})
   + \tfrac{\epsilon}{6} \\[0.5em]
&< \tfrac{\epsilon}{3}.
\end{align*}

On the other hand, by Condition \ref{it:F4}, we have: 
\begin{align*}
|\nu_{y_{\alpha_{\gamma}}}(U\setminus \{y_{\alpha_{\gamma}}\},V\setminus \{z_{\alpha_{\gamma}}\})|&\leq \mathrm{Var}(\nu_{y_{\alpha_{\gamma}}})(U\setminus \{y_{\alpha_{\gamma}}\},V\setminus \{z_{\alpha_{\gamma}}\})\\
&\leq \mathrm{Var}(\nu_{y_{\alpha_{\gamma}}})((K\setminus J_{\alpha_{\gamma}}) \cup\{x_{\alpha_\gamma}\},K\setminus J_{\alpha_{\gamma}})\\
&\leq  \mathrm{Var}(\nu_{y_{\alpha_{\gamma}}})(K\setminus J_{\alpha_{\gamma}},K\setminus J_{\alpha_{\gamma}})+ \mathrm{Var}(\nu_{y_{\alpha_{\gamma}}})(\{x_{\alpha_\gamma}\},K\setminus J_{\alpha_{\gamma}})\\
&\leq\mathrm{Var}(\nu_{y_{\alpha_{\gamma}}})(K\setminus J_{\alpha_{\gamma}},K\setminus J_{\alpha_{\gamma}})+\frac{\epsilon}{6} 
\end{align*}
and recalling that  $\nu_{y_{\alpha_\gamma}}(x_{\alpha_\gamma},z_{\alpha_\gamma})=0$, we can write:
\begin{align*}\mathrm{Var}(\nu_{y_{\alpha_\gamma}})((U\setminus \{y_{\alpha_\gamma}\},\{z_{\alpha_\gamma}\})&\leq \mathrm{Var}(\nu_{y_{\alpha_\gamma}})((K\setminus J_{\alpha_\gamma})\cup \{x_{\alpha_\gamma}\},\{z_{\alpha_\gamma}\})\\
&\leq \mathrm{Var}(\nu_{y_{\alpha_\gamma}})(K\setminus J_{\alpha_\gamma},\{z_{\alpha_\gamma}\})+|\nu_{y_{\alpha_\gamma}}(x_{\alpha_\gamma},z_{\alpha_\gamma})|\\
&= \mathrm{Var}(\nu_{y_{\alpha}})(K\setminus J_{\alpha_\gamma},\{z_{\alpha_\gamma}\})
\end{align*}

Hence, since $|\nu_{y_{\alpha_\gamma}}(y_{\alpha_\gamma},z_{\alpha_\gamma})|\geq \epsilon$, the relations above and Condition \ref{it:F3} implies
\begin{align*} 
|\nu_{y_{\alpha_{\gamma}}}(U,V)|
&\geq |\nu_{y_{\alpha_{\gamma}}}(y_{\alpha_{\gamma}},z_{\alpha_{\gamma}})|
  - \Bigl(
    \bigl|\nu_{y_{\alpha_{\gamma}}}(\{y_{\alpha_{\gamma}}\},\,V\setminus\{z_{\alpha_{\gamma}}\})\bigr|
    + \bigl|\nu_{y_{\alpha_{\gamma}}}(U\setminus\{y_{\alpha_{\gamma}}\},\,\{z_{\alpha_{\gamma}}\})\bigr| \\
&\qquad\qquad\qquad\qquad\qquad\qquad\qquad
    + \bigl|\nu_{y_{\alpha_{\gamma}}}(U\setminus\{y_{\alpha_{\gamma}}\},\,V\setminus\{z_{\alpha_{\gamma}}\})\bigr|
  \Bigr)\\[0.4em]
&\geq \epsilon
  - \Bigl(
    \mathrm{Var}(\nu_{y_{\alpha_{\gamma}}})(\{y_{\alpha_{\gamma}}\},\,V\setminus\{z_{\alpha_{\gamma}}\})
    + \mathrm{Var}(\nu_{y_{\alpha_{\gamma}}})(U\setminus\{y_{\alpha_{\gamma}}\},\,\{z_{\alpha_{\gamma}}\})\\
&\qquad\qquad\qquad\qquad\qquad\qquad\qquad
    + \mathrm{Var}(\nu_{y_{\alpha_{\gamma}}})(U\setminus\{y_{\alpha_{\gamma}}\},\,V\setminus\{z_{\alpha_{\gamma}}\})
  \Bigr)\\[0.4em]
&\geq \epsilon
  - \Bigl(
    \mathrm{Var}(\nu_{y_{\alpha_{\gamma}}})(\{y_{\alpha_{\gamma}}\},\,K\setminus J_{\alpha_{\gamma}})
    + \mathrm{Var}(\nu_{y_{\alpha_{\gamma}}})(K\setminus J_{\alpha_{\gamma}},\,\{z_{\alpha_{\gamma}}\})\\
&\qquad\qquad\qquad\qquad\qquad\qquad\qquad
    + \mathrm{Var}(\nu_{y_{\alpha_{\gamma}}})(K\setminus J_{\alpha_{\gamma}},\,K\setminus J_{\alpha_{\gamma}})
    + \tfrac{\epsilon}{6}
  \Bigr)\\[0.4em]
&> \epsilon - \tfrac{\epsilon}{3} = \tfrac{2\epsilon}{3}.
\end{align*}
\normalsize

Now we notice that both nets $\{x_{\alpha_\gamma}\}_{\gamma\in \varGamma}$ and $\{y_{\alpha_\gamma}\}_{\gamma\in \varGamma}$ converge to $a$. Since  $\nu$ is weak$^*$-continuous and that $U$ and $V$ are clopen compact subsets of $L$, the preceding inequalities yield
\[
|\nu_a(U,V)| = \lim_{\gamma\to\infty}|\nu_{x_{\alpha_\gamma}}(U,V)| \leq \tfrac{\epsilon}{3}
  < \tfrac{2\epsilon}{3} \leq \lim_{\gamma\to\infty}|\nu_{y_{\alpha_\gamma}}(U,V)| = |\nu_a(U,V)|,
\]
which is a contradiction.

\end{proof}

\section{Proof of Main Theorem}
\label{Sec:Final}

We are now ready to establish the main result of this work. Our argument relies on the principal theorems from Section~\ref{Sec:Decomposition}, which are gradually employed in the process of dismantling a bilinear operator $G:C_0(L)\times C_0(L)\to C_0(L)$ until only its most trivial components remain. As in the previous section, $L$ denotes a collapsing space (see Definition~\ref{Def:Collapsing}), and $K = L \cup \{\infty\}$ stands for its one-point compactification. For any bilinear operator $G:C_0(L)\times C_0(L)\to C_0(L)$, we denote by $\hat{G}: K \to \mathcal{F}_0(K,K)$ its representing function as in Theorem~\ref{Thm:Representation}, that is, the function satisfying
\[
G(f,g)(t) = \sum_{x\in K} \left( \sum_{y \in K} \hat{G}_t(x,y)\, g(y) \right) f(x).
\]
\begin{proof}[Proof of Theorem \ref{Thm:Main}]
Let $G:C_0(L)\times C_0(L)\to C_0(L)$ be a bilinear operator. According to Theorem \ref{Thm:Vanish-B}, there exist $(a_z)_{z\in L},(b_z)_{z\in L}\in \ell_1(L)$ such that, for each $z\in L$, there exists $\gamma_z<\omega_1$  such that, for all $x,y\in L\setminus L_{\gamma_z}$,
\[
\hat{G}_{x}(z,y)=
\begin{cases}
a_z, & \text{if } x = y, \\
0,   & \text{if } x \neq y,
\end{cases}
\qquad
\hat{G}_{x}(y,z)=
\begin{cases}
b_z, & \text{if } x = y, \\
0,   & \text{if } x \neq y.
\end{cases}
\]

Define the cross-interaction operator $H:C_0(L)\times C_0(L)\to C_0(L)$ by
\[
H(f,g)(t)=\sum_{w\in L} \big(a_w f(t)g(w) + b_w f(w)g(t)\big),
\]
which is well-defined since $(a_w)_{w\in L},(b_w)_{w\in L}\in \ell_1(L)$. Set $F=G-H$.

Fix $z\in L$, and let $\gamma_z$ be as above, increasing it if necessary so that $z\in L_{\gamma_z}$. For $x,y\in L\setminus L_{\gamma_z}$ we have
\begin{align*}
\hat{F}_x(z,y) &= \hat{G}_x(z,y) - \hat{H}_x(z,y) = \hat{G}_x(z,y) - a_z \chi_y(x),\\
\hat{F}_x(y,z) &= \hat{G}_x(y,z) - \hat{H}_x(y,z) = \hat{G}_x(y,z) - b_z \chi_y(x),
\end{align*}         
and thus $\hat{F}_x(z,y) = \hat{F}_x(y,z) = 0$. 

By Theorem \ref{Thm:Vanish-D}, there exist $r\in \mathbb{R}$ and $\rho<\omega_1$ such that $\hat{F}_y(y,y)=r$ for all $y \in L\setminus L_\rho$. Define the scaled multiplication operator $M:C_0(L)\times C_0(L)\to C_0(L)$ by $M(f,g)(t)=r\, f(t)g(t)$, and set $S=G-H-M$.

Towards a contradiction, assume that $S$ does not have separable image. Then, for each $\alpha<\omega_1$ there exist $f_\alpha,g_\alpha\in C_0(L)$ such that $S(f_\alpha,g_\alpha)\notin C_0(L_\alpha)$. Hence, there exists $x_\alpha\in L\setminus L_\alpha$ with $S(f_\alpha,g_\alpha)(x_\alpha)\neq 0$. We form the sequence $\{x_\alpha\}_{\alpha<\omega_1}$ and assume, after passing to a suitable refinement, that $\{x_\alpha\}_{\alpha<\omega_1}$ is a sequence of pairwise distinct points.

For each $\alpha<\omega_1$, we have
\[
S(f_\alpha,g_\alpha)(x_\alpha)=\sum_{y\in K}\Big(\sum_{z\in K}\hat{S}_{x_\alpha}(y,z)g_\alpha(z)\Big)f_\alpha(y)\neq 0,
\]
so we may fix $y_\alpha,z_\alpha \in L$ with $\hat{S}_{x_\alpha}(y_\alpha,z_\alpha)\neq 0$ and form the sequences $\{y_\alpha\}_{\alpha<\omega_1}$ and $\{z_\alpha\}_{\alpha<\omega_1}$.

Suppose that $z_\alpha=z$ for uncountably many $\alpha$. After a simultaneous refinement, we would have $\hat{S}_{x_\alpha}(y_\alpha,z)\neq 0$ for each $\alpha<\omega_1$. If $\{y_\alpha : \alpha<\omega_1\}$ is countable, this contradicts Theorem~\ref{Thm:Vanish-A}. If uncountable, there exists $y_\alpha\in L\setminus L_{\gamma_z}$, with $\gamma_z<\omega_1$ such that $\hat{F}_x(y,z)=\hat{F}_x(z,y)=0$ for all $x,y\in L\setminus L_{\gamma_z}$. Since $\hat{M}_x(y,z)\neq 0$ if and only is $x=y=z$, we obtain
\[
\hat{S}_{x_\alpha}(y_\alpha,z)=\hat{F}_{x_\alpha}(y_\alpha,z)-\hat{M}_{x_\alpha}(y_\alpha,z)=0,
\]
a contradiction. By the same reasoning, there is no refinement with $y_\alpha=y$ for uncountably many $\alpha$.

Now we claim that there exists a simultaneous refinement such that $\{x_\alpha\}_{\alpha<\omega_1}$, $\{y_\alpha\}_{\alpha<\omega_1}$, and $\{z_\alpha\}_{\alpha<\omega_1}$ are sequences of pairwise distinct points. Indeed, proceeding by transfinite induction on $\gamma<\omega_1$, assume that sequences $\{x_{\alpha_\xi}\}_{\xi<\gamma}$, $\{y_{\alpha_\xi}\}_{\xi<\gamma}$, $\{z_{\alpha_\xi}\}_{\xi<\gamma}$ of pairwise distinct points are obtained. Let
\[
\Lambda_y=\bigcup_{\xi<\gamma}\{\alpha<\omega_1:y_\alpha=y_{\alpha_\xi}\}, \quad 
\Lambda_z=\bigcup_{\xi<\gamma}\{\alpha<\omega_1:z_\alpha=z_{\alpha_\xi}\}.
\]
These are countable according to the discussion above, so fix $\alpha_\gamma<\omega_1$ strictly greater than $\sup(\Lambda_y \cup \Lambda_z \cup \{\alpha_\xi : \xi<\gamma\})$. This completes the inductive step, which establishes our claim.

Next, assuming that $\{x_\alpha\}_{\alpha<\omega_1}$, $\{y_\alpha\}_{\alpha<\omega_1}$, and $\{z_\alpha\}_{\alpha<\omega_1}$ are sequences of pairwise distinct points., if $x_\alpha\notin\{y_\alpha,z_\alpha\}$ for uncountably many $\alpha$, we contradict Theorem~\ref{Thm:Vanish-C}. Hence, after another simultaneous refinement, either $x_\alpha=y_\alpha$ for all $\alpha<\omega_1$ or $x_\alpha=z_\alpha$ for all $\alpha<\omega_1$. Without loss of generality assume the former. If $z_\alpha\neq y_\alpha$ for uncountably many $\alpha$, we contradict Theorem~\ref{Thm:Vanish-E}. Therefore, after a final simultaneous refinement, we may assume $x_\alpha=y_\alpha=z_\alpha$ for all $\alpha$. But then, for $y_\alpha\in L\setminus L_\rho$ (with $\rho$ fixed above), 
\[
\hat{S}_{y_\alpha}(y_\alpha,y_\alpha)=\hat{F}_{y_\alpha}(y_\alpha,y_\alpha)-\hat{M}_{y_\alpha}(y_\alpha,y_\alpha)=r-r=0,
\]
which is a contradiction. Thus, $S$ has separable image.

We now establish the uniqueness of the decomposition. Suppose there exist $r'\in\mathbb{R}$, $(a'_w)_{w\in L},(b'_w)_{w\in L}\in\ell_1(L)$, and a bilinear operator $S':C_0(L)\times C_0(L)\to C_0(L)$ with separable image such that
\[
G(f,g)(t)=r'f(t)g(t)+\sum_{w\in L}\big(a'_w f(t)g(w)+b'_w f(w)g(t)\big)+S'(f,g)(t),
\]
giving another representation of $G$. By Proposition \ref{Prop:SeparableImage}, there exists $\eta<\omega_1$ such that the images of $S$ and $S'$ lie in $C_0(L_\eta)$. For every $s\in L\setminus L_\eta$, the bilinear form $G_s:C_0(L)\times C_0(L)\to \mathbb{R}$, defined by $G_s(f,g)=G(f,g)(s)$, admits the following decompositions:
\begin{align*}
G(f,g)(s)&=r\, f(s)g(s)+f(s)\Big(\sum_{w\in L} a_w g(w)\Big)
         + g(s)\Big(\sum_{w\in L} b_w f(w)\Big)\\
         &=r^{\prime}\, f(s)g(s)+f(s)\Big(\sum_{w\in L} a^\prime_w g(w)\Big)
         + g(s)\Big(\sum_{w\in L} b^\prime_w f(w)\Big),
\end{align*}
that can be extended to $\ell_\infty(L)\times \ell_\infty(L)$, as shown in Theorem \ref{Thm:RieszBilinearScattered}. Now let $t\in L$ be arbitrary and choose $s\in L\setminus L_\eta$, with $s\neq t$. By computing $G(\chi_t,\chi_s)(s)$ and $G(\chi_s,\chi_t)(s)$ in both representations, we obtain $a_t=a^\prime_t$ and $b_t=b^\prime_t$. Hence, $(a_w)_{w\in L}=(a^\prime_w)_{w\in L}$ and $(b_w)_{w\in L}=(b^\prime_w)_{w\in L}$. Finally, computing 
$G(\chi_s,\chi_s)(s)$ for some $s\in L\setminus L_\eta$ in both representations gives $r=r^\prime$. This allows us to conclude that $S=S^\prime$, which completes the proof of the theorem.
\end{proof}

\section{Acknowledgements}

This research was supported by the Fundação de Amparo à Pesquisa do Estado de São Paulo (FAPESP), grant no.~2023/12916-1.

\bibliographystyle{amsalpha}

\end{document}